\documentclass[12pt]{article}

\usepackage{verbatim}

\usepackage{amsfonts}

\usepackage{amsthm}

\usepackage{amssymb}

\usepackage{amsmath}

\usepackage{mathabx}

\usepackage{enumerate}

\usepackage[all]{xy}

\usepackage{graphicx}

\usepackage{hyperref}

\newcommand{\N}{\mathbb{N}}

\newcommand{\Z}{\mathbb{Z}}

\newcommand{\R}{\mathbb{R}}

\newcommand{\C}{\mathbb{C}}

\newcommand{\h}{H}

\newcommand{\Hawaii}{Hawai\kern.05em`\kern.05em\relax i}

\newcommand{\red}{\textrm{red}}

\theoremstyle{plain}
\newtheorem{theorem}{Theorem}[section]
\newtheorem{lemma}[theorem]{Lemma}
\newtheorem{corollary}[theorem]{Corollary}
\newtheorem{proposition}[theorem]{Proposition}

\newtheorem{definition-theorem}[theorem]{Definition / Theorem}

\newtheorem*{conjecture*}{Conjecture}
\newtheorem*{theorem*}{Theorem}

\theoremstyle{definition}
\newtheorem{definition}[theorem]{Definition}
\newtheorem{example}[theorem]{Example}

\newtheorem{convention}[theorem]{Convention}

\theoremstyle{remark}
\newtheorem{remark}[theorem]{Remark}

\newtheorem{examples}[theorem]{Examples}

\newtheorem*{example*}{Example}  
\newtheorem*{remark*}{Remark}

\title{Topological property (T) for groupoids}

\author{Cl\'{e}ment Dell'Aiera and Rufus Willett}

\begin{document}

\maketitle

\begin{abstract}
We introduce a notion of topological property (T) for \'{e}tale groupoids.  This simultaneously generalizes Kazhdan's property (T) for groups and geometric property (T) for coarse spaces.  One main goal  is to use this property (T) to prove the existence of so-called Kazhdan projections in both maximal and reduced groupoid $C^*$-algebras, and explore applications of this to exactness, $K$-exactness, and the Baum-Connes conjecture.  We also study various examples, and discuss the relationship with other notions of property (T) for groupoids and with a-T-menability.  
\end{abstract}

\tableofcontents

\section{Introduction}

Property (T) is an important rigidity property of groups introduced by Kazhdan \cite{Kazhdan:1967aa}, and much studied for its applications and connections to several parts of mathematics: see for example the monograph \cite{Bekka:2000kx} for an overview and historical comments.  Property (T) has also been extended to measured groupoids by Zimmer \cite{Zimmer:2981ds} (for equivalence relations) and Anantharaman-Delaroche \cite{Anantharaman-Delaroche:2005pb} (in a fairly general setting).  Measured property (T) has very interesting connections to von Neumann algebra theory via the construction of groupoid von Neumann algebras, and in particular to the special case of group actions via the group measure-space construction: see for example \cite{Lupini:2017aa} and the references given there.

For applications to groupoid $C^*$-algebras, one needs a topological version of property (T) for groupoids, and this currently seems to be missing from the literature.  It is the goal of this paper to give one possible definition that fills this gap, particularly motivated by work of Higson, Lafforgue, and Skandalis \cite{Higson:2002la}.  Indeed, these authors were able to show that certain projections in groupoid $C^*$-algebras have bad properties from the point of view of exactness, and thus to produce counterexamples to versions of the Baum-Connes conjecture.  The projections constructed by Higson, Lafforgue, and Skandalis have a lot in common with the so-called Kazhdan projections in group $C^*$-algebras first constructed by Akemann-Walter \cite{Akemann:1981dz} using property (T).  This analogy is particularly good when one uses the approach to these projections exploiting spectral gap phenomena due to Valette \cite[Theorem 3.2]{Valette:1984wy} and as extensively studied recently by Drutu and Nowak \cite{Drutu:2015aa}.  

From the above discussion, it seems natural to try to define a topological version of property (T) that works for groupoids, and allows one to construct such Kazhdan projections in associated groupoid $C^*$-algebras.  Indeed, this was implicitly done by the second author and Yu \cite{Willett:2013cr} in a special case.  These authors introduced a notion called geometric property (T) for coarse spaces; moreover, geometric property (T) can be interpreted as a property of the associated coarse groupoid introduced by Skandalis, Tu and Yu in \cite{Skandalis:2002ng}.  Another motivation of ours was to generalize geometric property (T) from coarse groupoids to a more general class of groupoids.

There is something a little mysterious about the Kazhdan projections considered (at least implicitly) by Higson, Lafforgue, and Skandalis when compared to the group case.  In the group case, Kazhdan projections live in the maximal group $C^*$-algebra $C^*_{\max}(G)$, but (other than in the very special situation where the underlying group is compact) must map to zero in the reduced group $C^*$-algebra $C^*_r(G)$.  However, in the groupoid case, there can be Kazhdan projections that are non-zero in both $C^*_{\max}(G)$ and $C^*_r(G)$, or even that are non-zero in $C^*_r(G)$ without existing in $C^*_{\max}(G)$.  These sort of phenomena are crucial for the work of Higson, Lafforgue, and Skandalis: the Baum-Connes conjecture is about the $K$-theory of $C^*_r(G)$, so one needs projections in the reduced $C^*$-algebra.  An important motivation for us was to clarify all this; although it would be a little unwieldy to give details in this introduction, let us say that the existence of non-trivial Kazhdan projections in $C^*_r(G)$ has to do with interactions between the parts of the base space that emit finitely many arrows, and those parts that emit infinitely many.

\subsection*{Outline}

Although studying Kazhdan projections is our main motivation, we expect that topological property (T) for groupoids will have other interesting applications just as in the group case, and take the opportunity to develop some basic theory.  Thus having gone over some conventions in Section \ref{con sec}, we start by giving an account of what we mean by property (T) for groupoids in Section \ref{basic sec}: much as in the group case, the basic idea is that invariant vectors in representations must be isolated from the rest in some appropriate sense.  In the groupoid case, however, there are at least two reasonable definitions of invariant vector, so there are some foundational issues about this to consider before one can even get started; this is all done in Section \ref{basic sec}.  We then discuss some natural classes of examples in Section \ref{ex sec}, including connections to coarse geometry, group actions, and property $(\tau)$.  In Section \ref{meas sec} we discuss the relationship of our notion to other definitions of property (T) for groupoids, including the work of Zimmer and Anantharaman-Delaroche in the measured setting that was mentioned above.  In Section \ref{atmen sec}, we discuss the relationship with a-T-menability for groupoids as defined by Tu \cite[Section 3]{Tu:1999bq}; as one might expect by analogy with the group case, property (T) is incompatible with a-T-menability at least in some cases.  In Section \ref{kaz sec} we finally get back to our main motivation and give a fairly thorough discussion of the existence of Kazhdan projections in groupoid $C^*$-algebras and applications to exactness, $K$-exactness, and the Baum-Connes conjecture.  Finally, in Section \ref{q sec}, we summarize some open questions.

This paper is fairly long, and we expect different parts might interest different audiences.  We have thus aimed to write the paper in a fairly modular way: after Section \ref{basic sec}, it should be possible to read any of Sections \ref{ex sec}, \ref{meas sec}, \ref{atmen sec} and \ref{kaz sec} more-or-less independently of the others.

\subsection*{Acknowledgments}

As mentioned above, this work is partly an attempt to generalize joint work of the second author with Guoliang Yu in the coarse geometric setting.  The authors are grateful to Professor Yu for several interesting conversations around this subject.  We are also grateful to Jesse Peterson for pointing out some references, and other interesting comments.  Finally, we would like to thank Kenny Corea and the anonymous referee for pointing out some mistakes, and suggesting improvements.

The authors were partly supported by the US NSF (DMS 1564281 for the first author, and DMS 1401126 and DMS 1564281 for the second author).

\section{Conventions}\label{con sec}

As there is some inconsistency about notational and terminological conventions in the groupoid $C^*$-algebra literature\footnote{And indeed, even between our own papers!}, we list ours here.  For background on the class of groupoids we consider and the associated $C^*$-algebras, we recommend \cite[Section 2.3]{Renault:2009zr}, \cite[Section 5.6]{Brown:2008qy}, and \cite{Sims:2017aa}; see these references for precise definitions of the various objects we introduce below.

Groupoids will be denoted $G$, with \emph{base space} or \emph{unit space} $G^{(0)}$, which we identify with a subset of $G$.   Typically, we write elements of $G$ using letters like $g,h,k$, and elements of $G^{(0)}$ using letters like $x,y,z$.  An ordered pair $(g,h)\in G\times G$ is \emph{composable} if $s(g)=r(h)$, in which case we write $gh$ for their product.  For $x\in G^{(0)}$, the \emph{range fibre} and \emph{source fibre} are of $x$ are defined by 
$$
G^x:=r^{-1}(x)\quad\text{and}\quad G_x:=s^{-1}(x)
$$
respectively.  If $A$, $B$ are two subsets of $G$, we define 
$$
G_A^B:=\{g\in G\mid s(g)\in A\text{ and } r(g)\in B\}.
$$
We define also 
$$
A^{-1}:=\{g^{-1}\mid g\in G\} \quad \text{and}\quad AB:=\{gh\mid g\in A,h\in B \text{ and } s(g)=r(h)\}
$$
(note that $AB$ could be empty even if $A$ and $B$ are not).

A groupoid will always be assumed to be equipped with a locally compact, Hausdorff topology. 
We will always assume that the inverse and composition maps are continuous.  A \emph{bisection} is an open subset $B$ of $G$ on which $r$ and $s$ restrict to homeomorphisms.  We will always assume that $G$ is \emph{\'{e}tale}, meaning that there is a basis for its topology consisting of open bisections; note that this implies that $r$ and $s$ are continuous and open maps, that $G^{(0)}$ is closed and open in $G$, and that each $G_x$ and $G^x$ are discrete in the subspace topology.

We will sometimes need to use measures on $G$ and $G^{(0)}$.  A \emph{measure} on a locally compact Hausdorff space $X$ will  always mean a \emph{Radon measure}, i.e.\ a positive element $\mu:C_c(X)\to \C$ of the continuous dual of the topological vector space $C_c(X)$ of continuous compactly supported complex-valued functions on $X$; we will also think of measures as appropriate maps $\mu:\mathcal{B}(X)\to [0,\infty]$ from the collection of Borel subsets of $X$ to $[0,\infty]$ when convenient.  A measure is a \emph{probability measure} if $\mu(X)=1$.  

Given a measure $\mu$ on $G^{(0)}$, define measures $r^*\mu$ and $s^*\mu$ on $G$ as functionals on $C_c(G)$ via the formulas 
$$
(r^*\mu)(f):=\int_{G^{(0)}}\sum_{g\in G^x}f(g)d\mu(x)\quad \text{and}\quad (s^*\mu)(f):=\int_{G^{(0)}}\sum_{g\in G_x}f(g)d\mu(x).
$$
A measure $\mu$ on $G^{(0)}$ is \emph{quasi-invariant} if $r^*\mu$ and $s^*\mu$ have the same null sets, in which case the associated \emph{modular function} $D:G\to (0,\infty)$ is defined to be the Radon-Nikodym derivative $D:=d(r^*\mu) / d(s^*\mu)$.  A measure on $G^{(0)}$ is \emph{invariant} if $r^*\mu=s^*\mu$, or equivalently, if $\mu(r(B))=\mu(s(B))$ for any Borel bisection $B$.

The \emph{convolution $*$-algebra} of $G$ identifies as a vector space with the space $C_c(G)$ of continuous, compactly supported, complex-valued functions on $G$.  The multiplication and adjoint operations on $C_c(G)$ are defined by 
$$
(f_1f_2)(g):=\sum_{hk=g}f_1(h)f_2(k)\quad \text{and} \quad f^*(g):=\overline{f(g^{-1})}
$$
respectively.  The maximal and reduced $C^*$-algebraic completions of $C_c(G)$ will be denoted by $C^*_{\max}(G)$ and $C^*_{r}(G)$ respectively.  In addition to the reduced and maximal $C^*$-norms on $C_c(G)$, we will need the \emph{$I$-norm} defined for $f\in C_c(G)$ by
$$
\|f\|_I:=\max\Bigg\{\sup_{x\in G^{(0)}}\sum_{g\in G^x}|f(g)|,~\sup_{x\in G^{(0)}}\sum_{g\in G_x}|f(g)|\Bigg\}.
$$
 
A \emph{representation} of $C_c(G)$ is by definition a $*$-homomorphism 
$$
\pi:C_c(G)\to \mathcal{B}(\h)
$$
from $C_c(G)$ to the $C^*$-algebra of bounded operators on some Hilbert space $\h$; our Hilbert spaces are always complex, and inner products are linear in the second variable.  Typically we write $(\h,\pi)$ for a representation.  Often, we will leave the map $\pi$ implicit in the notation unless this seems likely to cause confusion, writing for example `$f\xi$' rather than `$\pi(f)\xi$' for $f\in C_c(G)$ and $\xi\in \h$.  Note that any representation of $C_c(G)$ extends uniquely\footnote{In the literature this is often stated as a consequence of Renault's disintegration theorem, and thus something that is only known to hold in the second countable case; however, for \'{e}tale groupoids it is always true, and not difficult to prove directly.  See for example \cite[Theorem 3.2.2]{Sims:2017aa}.} to a representation of $C^*_{\max}(G)$, i.e.\ to a $*$-homomorphism 
$$
\pi:C^*_{\max}(G)\to \mathcal{B}(H),
$$
and any such representation restricts to a unique representation of $C_c(G)$; as such, we will sometimes identify representations of $C_c(G)$ with representations of $C^*_{\max}(G)$.

As this is certainly not universal, we finish this section by emphasizing the following convention.

\begin{convention}\label{gpd conv}
Throughout this paper, all groupoids are assumed to be locally compact, Hausdorff, \'{e}tale, and to have compact unit space (other than in a few side remarks).  We will generally not repeat these assumptions; thus in this paper \textbf{``groupoid'' means locally compact, Hausdorff, \'{e}tale groupoid with compact unit space}.
\end{convention}

Much of what we do could be carried out in more generality; we make a few comments below about possible generalisations where we feel this might be useful.  However, we thought it would be better to keep to a relatively simple setting so as not to lose the main ideas in excessive technicalities, and also as our assumptions cover the examples that we are most interested in.

\section{Constant vectors and property (T)}\label{basic sec}

See Convention \ref{gpd conv} for our use of the word ``groupoid''. 

In this section, we introduce our notion of property (T) for groupoids (as usual, locally compact, Hausdorff, \'{e}tale, and with compact unit space).  Just like property (T) for groups, the idea is that the `constant vectors' in any representation of $C_c(G)$ should be isolated in some sense.  

However, unlike for groups it is not completely clear what a constant vector in a representation of $C_c(G)$ should mean: there seem to be at least two genuinely different reasonable definitions.  The definition below is well-suited to our applications. 

\begin{definition}\label{con vec def}
Let $G$ be a groupoid.  Define a linear map by
$$
\Psi:C_c(G)\to C(G^{(0)}),\quad f\mapsto \sum_{g\in G^x}f(g).
$$
Note that the image is indeed contained in $C(G^{(0)})$: indeed, it suffices by the \'{e}tale assumption to check this for $f$ supported in an open bisection, in which case it is clear.  For a representation $(\h,\pi)$ of $C_c(G)$, a vector $\xi\in \h$ is \emph{invariant}, or \emph{fixed}, or \emph{constant} if for all $f\in C_c(G)$,
$$
f\xi=\Psi(f)\xi.
$$
We write $\h^\pi$ for the closed subspace of $\h$ consisting of constant vectors, and $\h_\pi$ for its orthogonal complement.  
\end{definition}

In order to fix intuition, let us look at some examples.

\begin{example}\label{gp ex}
Let $G=\Gamma$ be a discrete group, so $C_c(G)=\C[\Gamma]$ is the usual complex group $*$-algebra, with elements given by formal sums $\sum_{g\in\Gamma}a_gg$ with finitely many non-zero complex coefficients $a_g\in \C$.  Representations of $C_c(G)$ are canonically in one-to-one correspondence with unitary representations of $\Gamma$.  Moreover, $C_c(G^{(0)})=\C$, and $\Psi(\sum a_gu_g)=\sum a_g$.  From this, one sees that in any representation $(H,\pi)$ of $C_c(G)$, a vector $\xi$ is fixed if and only if it is fixed by the corresponding unitary representation $u$ of $\Gamma$, i.e.\ if and only if $u_g\xi=\xi$ for all $g\in \Gamma$.
\end{example}

\begin{example}\label{triv rep}
Let $G$ be a groupoid, and let $\mu$ be an invariant probability measure on $G^{(0)}$.  Let $H_\mu$ be the Hilbert space $L^2(G^{(0)},\mu)$, and define a representation $\tau_\mu$ of $C_c(G)$ on $H$ by the formula
$$
(\tau_\mu(f)\xi)(x):=\sum_{g\in G^x} f(g)\xi(s(g)).
$$
The pair $(H_\mu,\tau_\mu)$ is called the \emph{trivial representation} associated to $\mu$.  Then any function $\xi:G^{(0)}\to \C$ that is constant in the usual sense is invariant for $\tau_\mu$.   More generally, $\xi\in H_\mu$ is invariant if and only if for $\mu$-almost-every $x\in G^{(0)}$ and every $g\in G^x$, $\xi(x)=\xi(s(g))$ (roughly, `$\xi$ is constant on almost every orbit').
\end{example}

Note that the above example shows that $H^\pi$ and $H_\pi$ will \emph{not} be invariant under $\pi$ in general, and therefore (unlike the group case), the constant vectors do not define a subrepresentation of $(H,\pi)$ in general.  

The above example of constant vectors is in some sense general.  The next proposition formalises this; we include it mainly for intuition, and will not really use it in the rest of the paper.

\begin{proposition}\label{inv vec char}
Let $G$ be a groupoid, $(H,\pi)$ be a representation of $C_c(G)$, and $\xi\in H^\pi$ be a constant vector.  Then the measure $\mu_\xi$ on $G^{(0)}$ defined by 
$$
\mu_\xi:C(G^{(0)})\to \C, \quad f\mapsto \langle \xi, f\xi\rangle
$$
is invariant.  

Moreover, the cyclic subrepresentation of $(H,\pi)$ generated by $\xi$ is unitarily equivalent to the trivial representation $(H_{\mu_\xi},\tau_{\mu_\xi})$ of Example \ref{triv rep} via a unitary isomorphism that takes $\xi$ to the constant function with value one.
\end{proposition}

\begin{proof}
Recall that a measure $\mu$ on $G^{(0)}$ is invariant if and only if $r^*\mu=s^*\mu$, i.e.\ if and only if 
$$
\int_{G^{(0)}}\sum_{g\in G^x}f(g)d\mu(x)=\int_{G^{(0)}}\sum_{g\in G_x}f(g)d\mu(x)
$$
for all $f\in C_c(G)$.  In the case $\mu=\mu_\xi$, note that the left hand side equals $\langle \xi,\Psi(f)\xi\rangle$ by definition of $\mu_\xi$, and the right hand side equals $\langle \Psi(f^*)\xi,\xi\rangle$.  Hence to show invariance of $\mu_\xi$, we must show that 
$$
\langle \xi,\Psi(f)\xi\rangle=\langle \Psi(f^*)\xi,\xi\rangle
$$
for all $f\in C_c(G)$.  However, invariance of $\xi$ gives
$$
\langle \xi,\Psi(f)\xi\rangle=\langle \xi,f\xi\rangle=\langle f^*\xi,\xi\rangle=\langle \Psi(f^*)\xi,\xi\rangle
$$
as required.

For the unitary equivalence statement, we compute that for any $f\in C_c(G)$,
$$
\langle \xi,f\xi\rangle_H=\langle \xi,\Psi(f)\xi\rangle_H=\int_{G^{(0)}}\sum_{g\in G^x}f(g)d\mu_{\xi}(x)=\langle 1,\tau_{\mu_\xi}(f)1\rangle_{H_{\mu_\xi}}.
$$
Hence the unitary equivalence statement follows from the uniqueness of $*$-representations of an involutive algebra with specified cyclic vector (see for example \cite[Proposition 2.4.1]{Dixmier:1977vl}). 
\end{proof}

The following corollary is immediate.  It shows in particular that for many groupoids, $C_c(G)$ does not admit any representations with non-zero constant vectors.  This is in sharp contrast to the group case where such representations always exist.

\begin{corollary}\label{inv meas}
A groupoid $G$ admits a representation with non-zero constant vectors if and only if $G^{(0)}$ admits an invariant probability measure. \qed
\end{corollary}

We are now ready to give our definition of property (T).  

\begin{definition}\label{basic t}
Let $G$ be a groupoid.  A subset $K$ of $G$ is a \emph{Kazhdan set} if there exists $c>0$ such that for any representation $(H,\pi)$ of $C_c(G)$ and any $\xi\in H_\pi$, there exists $f\in C_c(G)$ with support in $K$ and $\|f\|_I\leq 1$ such that $\|f\xi-\Psi(f)\xi\|\geq c\|\xi\|$.

The groupoid $G$ has \emph{topological property (T)} if it admits a compact Kazhdan set.
\end{definition}

We will generally just say `property (T)', omitting the word `topological' unless we need to make a distinction with the measure-theoretic case.  If $K$ is a Kazhdan set for $G$ and $c>0$ satisfies the condition in Definition \ref{basic t}, then $(K,c)$ will be called a \emph{Kazhdan pair}, and $c$ will be called a \emph{Kazhdan constant}.   We will give examples in the next section. 

We will also be interested in the following family of weaker variants of property (T).

\begin{definition}\label{prop t def}
Let $G$ be a groupoid, and let $\mathcal{F}$ be a class of representations of $C_c(G)$.  A subset $K$ of $G$ is a \emph{Kazhdan set for $\mathcal{F}$} if there exists $c>0$ such that for any representation $(H,\pi)$ of $C_c(G)$ in the collection $\mathcal{F}$, and any $\xi\in H_\pi$, there exists $f\in C_c(G)$ with support in $K$ and $\|f\|_I\leq 1$ such that $\|f\xi-\Psi(f)\xi\|\geq c\|\xi\|$.

The groupoid $G$ has \emph{(topological) property (T) with respect to $\mathcal{F}$} if it admits a compact Kazhdan set. 
\end{definition}

We will again talk about Kazhdan pairs and constants with respect to $\mathcal{F}$ in the obvious ways.  

Note that property (T) as in Definition \ref{basic t} is the same as property (T) for the family of all representations of $C_c(G)$.  In general, the larger $\mathcal{F}$ is, the stronger a condition having property (T) with respect to $\mathcal{F}$ is, so property (T) itself is the strongest variant.

We will be particularly interested in the following example of a family of representations.

\begin{example}\label{reg reps}
For $x\in G^{(0)}$, the \emph{regular representation} of $C_c(G)$ associated to $x$ is the pair $(\ell^2(G_x),\pi_x)$, where 
$$
(\pi_x(f)\xi)(g):=\sum_{h\in G_x}f(gh^{-1})\xi(h)
$$
for $f\in C_c(G)$ and $\xi\in \ell^2(G_x)$ (compare \cite[Section 2.3.4]{Renault:2009zr}).  We denote the family of all such representations by $\mathcal{F}_r$.  This family is particularly interesting as the reduced $C^*$-algebra $C^*_r(G)$ is (by definition) the completion of $C_c(G)$ for the norm 
$$
\|f\|_r:=\sup_{x\in G^{(0)}}\|\pi_x(f)\|_{\mathcal{B}(\ell^2(G_x))}.
$$
\end{example}

Let us conclude this section with a remark on possible generalisations.

\begin{remark}\label{prop t gen rem}
There are several natural generalizations of the definition of property (T) above.  We sketch some of these out here; we would be very happy if someone else explores these in future work.

One could consider more general locally compact groupoids with Haar system (and compact unit space).  Having replaced the sum by an integral with respect to Haar measure in the definition of $\Psi:C_c(G)\to C(G^{(0)})$ (Definition \ref{con vec def}), everything else makes sense in this level of generality.  It would also be natural to expand the definition to cover non-compact base spaces.  For this it seems most reasonable to proceed as follows: say that a subset $E$ of a groupoid $G$ is \emph{fibrewise compact} if for any compact subset $K$ of $G^{(0)}$, $E\cap G_K^K$ is compact.  Then define property (T) for a groupoid with possibly non-compact base space to mean that there exists a fibrewise compact Kazhdan set.  

Another natural generalisation would be to look at broader classes of representations of $C_c(G)$: for example, Hilbert space representations that are not $*$-representations, or representations on suitable classes of Banach spaces.  Indeed, there has been a great deal of relatively recent interesting work in the group case in these settings: for example \cite{Bader:2007eg,Lafforgue:2007fj,Drutu:2015aa}.  

As for the analogues in the group case, we expect these generalizations would be interesting.  We did not pursue any of these seriously mainly to keep the current paper down to a relatively reasonable length, and minimize our discussion of technical issues.
\end{remark}

\section{Examples}\label{ex sec}

In this section, we discuss some basic examples of groupoids with property (T).  We remind the reader that our groupoids are always locally compact, Hausdorff, \'{e}tale, and have compact base space as in Convention \ref{gpd conv}.  We will not repeat these assumptions in the body of the section.

\subsection{Trivial and compact groupoids}

The most basic class of groupoids with property (T) are the trivial groupoids, i.e.\ those for which $G=G^{(0)}$.  Indeed, in this case for any representation $(H,\pi)$ of $C_c(G)$, $H^\pi=H$, so the definition is vacuous.  

The second most basic class probably consists of compact groupoids as in the next result.

\begin{proposition}\label{com gpd t}
Any compact groupoid has property (T).
\end{proposition}

\begin{proof}
We claim that $G$ itself is a Kazhdan set, with associated Kazhdan constant one.  Indeed, let $\chi:G\to\C$ be the constant function with value one everywhere, and let $p=\chi / (\Psi(\chi)\circ r)$.  Then one checks directly that $p$ is a well-defined element of $C_c(G)$, that for all $f\in C_c(G)$ we have $fp=\Psi(f)p$ (here the products are convolution in $C_c(G)$), that $p=p^*$, and that $\Psi(p)=1$.  From these computations it follows also that $p^2=p$, and that the image of $p$ in any representation of $C_c(G)$ is exactly the orthogonal projection onto the constant vectors.  Hence for any representation $(\h,\pi)$, and any $\xi\in \h_\pi$ we have that 
$$
\|p\xi-\Psi(p)\xi\|=\|0-\xi\|\geq \|\xi\|,
$$
which gives the desired conclusion.
\end{proof}

\subsection{Groups}

In this subsection we show that our version of property (T) reduces to the usual one for discrete groups (i.e.\ groups that are \'{e}tale when considered as groupoids).  

The following definition is taken from \cite[Definition 1.1.3]{Bekka:2000kx}.  For a Hilbert space $H$, let $\mathcal{U}(H)$ denote the unitary group of $H$.

\begin{definition}\label{gpt}
Let $G$ be a discrete group, and let 
$$
u:G\to \mathcal{U}(\h)
$$
be a unitary representation of $G$.  A vector $\xi\in \h$ is \emph{constant} if $u_g\xi=\xi$ for all $g\in G$.

A subset $S$ of $G$ is a \emph{Kazhdan set} if there exists $c>0$ such that if $(H,u)$ is a unitary representation of $G$ such that 
$$
\|u_g\xi-\xi\|<c\|\xi\|
$$
for all $g\in S$, then there exists a non-zero invariant vector in $H$.

The group $G$ has \emph{property (T)} if it admits a finite Kazhdan set.
\end{definition}

We now have two definitions of `Kazhdan set' for groups: Definition \ref{gpt} and the specialisation of Definition \ref{prop t def}.  Temporarily, if $G$ is a discrete group let us say a \emph{group Kazhdan set} a Kazhdan set in the sense of Definition \ref{prop t def} and a \emph{groupoid Kazhdan set} a Kazhdan set in the sense of Definition \ref{prop t def}, and similarly for the notions of invariant vector.

\begin{proposition}\label{2kaz}
Let $G$ be a discrete group.  Then a finite subset of $G$ is a group Kazhdan set if and only if it is a groupoid Kazhdan set.
\end{proposition}

\begin{proof}
Assume first that $K$ is a groupoid Kazhdan set with associated Kazhdan constant $c>0$.  Let $u:G\to \mathcal{U}(\h)$ be a unitary representation of $G$ and $\xi\in H$ be such that $\|u_g\xi-\xi\|<c\|\xi\|$ for all $g\in K$.  Denote by $\pi$ the usual extension of $u$ to $C_c(G)=\C[G]$ defined by 
$$
\pi:\sum_{g\in G}a_g g \mapsto \sum_{g\in G} a_g u_g.
$$
Letting $f=\sum f(g)g\in C_c(G)$ be supported in $K$ with $\|f\|_I\leq 1$, we see that with $\xi$ as above,
$$
\|\pi(f)\xi-\pi(\Psi(f))\xi\|\leq \sum_{g\in G}|f(g)|\|u_g\xi-\xi\|<c\|f\|_I\sup_{g\in K}\|u_g\xi-\xi\|\leq c\|\xi\|.
$$
As $(K,c)$ is a groupoid Kazhdan pair, this forces $H\neq H_\pi$, and so $H^\pi\neq \{0\}$, and $K$ is a group Kazhdan set.

Conversely, say $S$ is a group Kazhdan set with associated Kazhdan constant $c>0$.  Let $\pi:C_c(G)\to \mathcal{B}(\h)$ be a representation of $C_c(G)$, and let $u$ be the associated unitary representation of $G$ defined by $u_g=\pi(\chi_{\{g\}})$, where $\chi_{\{g\}}$ is the characteristic function of the singleton $\{g\}$.  Let $\xi$ be a vector in $\h_\pi$, and note that as $u$ leaves $H^\pi$ invariant it restricts to a representation on $\h_\pi$.  As $H_\pi$ has no invariant vectors and as $S$ is a group Kazhdan set, there exists $g\in S$ with $\|u_g\xi-\xi\|\geq c\|\xi\|$.  Then the function $f=\chi_{\{g\}}$ is supported in $S$, satisfies $\|f\|_I\leq 1$, and also that 
$$
\|f\xi-\Psi(f)\xi\|\geq c\|\xi\|.
$$ 
Hence $S$ is also a groupoid Kazhdan set.
\end{proof}

\begin{corollary}\label{gpgpd}
A discrete group has property (T) in the sense of Definition \ref{basic t} if and only if it has it in the sense of Definition \ref{gpt}.  \qed
\end{corollary}

\subsection{Coarse spaces}

Yu and the second author introduced a notion called \emph{geometric property (T)} in \cite{Willett:2013cr} for monogenic, bounded geometry coarse spaces.  On the other hand, Skandalis, Tu, and Yu \cite{Skandalis:2002ng} introduced a \emph{coarse groupoid} $G(X)$ associated to any bounded geometry coarse space $X$.  Our goal in this subsection is to explain why geometric property (T) for $X$ is equivalent to property (T) for $G(X)$.  This example is one of the main motivations behind our definition of property (T) for groupoids. 

Let $X$ be a coarse space as in \cite[Definition 2.3]{Roe:2003rw}.  Precisely, this means that $X$ is equipped with a collection $\mathcal{E}$ of subsets of $X\times X$ called \emph{controlled sets} which contains the diagonal, and is closed under the formation of subsets, finite unions, inverses, and products, where the inverse of $E$ is defined by 
$$
E^{-1}:=\{(x,y)\in E\mid (y,x)\in E\}
$$
and the product of two subsets $E$ and $F$ of $X\times X$ is defined to be 
$$
E\circ F:=\{(x,z)\in X\times X\mid \text{ there exists } y\in X \text{ with } (x,y)\in E \text{ and } (y,z)\in F\}.
$$
Such a collection $\mathcal{E}$ is called a \emph{coarse structure} on $X$.  A coarse structure has \emph{bounded geometry} if the suprema of cardinalities of `slices'
$$
\sup_{x\in X}|\{y\in X\mid (x,y)\in E\}|\quad \text{and}\quad \sup_{y\in X}|\{x\in X\mid (x,y)\in E\}|
$$
are both finite.  A controlled set $E$ \emph{generates} the coarse structure if $\mathcal{E}$ is the smallest coarse structure containing $E$, and a coarse structure $\mathcal{E}$ is \emph{monogenic} if a generator exists.

The \emph{uniform Roe $*$-algebra} of a bounded geometry, monogenic coarse space $X$, denoted $\C_u[X]$, consists of all $X$-by-$X$ matrices $a=(a_{xy})_{x,y\in X}$ with uniformly bounded complex entries, and such that the set $\{(x,y)\in X\times X\mid a_{xy}\neq 0\}$ is controlled.  The uniform Roe $*$-algebra is then a $*$-algebra when equipped with the usual matrix operations.  Following \cite[Section 3]{Willett:2013cr}, define a linear map
$$
\Phi:C_u[X]\to \ell^\infty(X),\quad \Phi(a)(x):=\sum_{y\in X}a_{xy}.
$$
A \emph{representation} of $\C_u[X]$ is by definition a $*$-representation as bounded operators on some Hilbert space.  If $(H,\pi)$ is such a representation, then a vector $\xi\in H$ is called \emph{constant} if $a\xi=\Phi(a)\xi$ for all $a\in \C_u[X]$.  We will denote the constant vectors in $H$ by $H_c$.

The following definition comes from \cite[Proposition 3.8]{Willett:2013cr}.

\begin{definition}\label{geo t def}
Let $X$ be a bounded geometry, monogenic coarse space.  Then $X$ has \emph{geometric property (T)} if for every generating controlled set $E$ there exists $c>0$ such that for every representation $(H,\pi)$ of $\C_u[X]$ and every vector $\xi\in H_c^\perp$ there exists $a\in C_u[X]$ with $\{(x,y)\in X\times X\mid a_{xy}\neq 0\}$ contained in $E$, and such that 
$$
\|a\xi-\Phi(a)\xi\|\geq c\sup_{x,y}|a_{xy}|\|\xi\|.
$$
\end{definition}

We now recall the definition of the coarse groupoid $G(X)$ from \cite{Skandalis:2002ng}; see also the expositions in \cite[Chapter 10]{Roe:2003rw} and \cite[Appendix C]{Spakula:2014aa}.  Let $\beta X$ be the Stone-\v{C}ech compactification of $X$.  For each controlled set $E$, let $\overline{E}$ be the closure of $E$ inside $\beta X\times \beta X$ for the natural inclusion $X\times X\subseteq \beta X\times \beta X$, which one can check is a compact open set.  Equip each $\overline{E}$ with the subspace topology. Define 
$$
G(X) := \bigcup_{E\in \mathcal{E}} \overline{E}
$$
equipped with the weak topology it inherits as the union of open subsets $\overline{E}$: precisely, this means that a subset $U$ of $G(X)$ is open precisely when $U\cap \overline{E}$ is open in $\overline{E}$ for all $E\in \mathcal{E}$ (this is \emph{not} the topology it inherits as a subspace of $\beta X\times \beta X$).  Equip $G(X)$ with the groupoid operations it inherits as a subset of the pair groupoid $\beta X\times\beta X$.  It is shown in \cite[Theorem 10.20]{Roe:2003rw} that $G(X)$ thus defined is a (locally compact, Hausdorff, \'{e}tale) groupoid, with base space $\beta X$.

\begin{proposition}\label{t geot}
For a monogenic bounded geometry coarse space $X$, geometric property (T) for $X$ and property (T) for $G(X)$ are equivalent.  
\end{proposition}

\begin{proof}
For $f\in C_c(G(X))$, note that $f$ restricts to a function on $X\times X$.  Define an element $a^f\in \C_u[X]$ by the formula $a^f_{xy}:=f(x,y)$.  It is proved in \cite[Proposition 10.28]{Roe:2003rw} that the map
$$
C_c(G(X))\to \C_u[X], \quad f\mapsto a^f
$$
is a $*$-isomorphism.  It is moreover not difficult to see that this map takes $C(\beta X)$ to $l^\infty(X)$, and that it `intertwines' $\Psi$ and $\Phi$ in the sense that 
$$
\Phi(a^f)=a^{\Psi(f)}.
$$ 
It follows from this that representations $(H,\pi)$ of $\C_u[X]$ and of $C_c(G(X))$ are in one-to-one correspondence, and that the two notions of constant vectors that we have defined using $\Phi$ and $\Psi$ correspond.  The remainder of the proof is essentially a translation exercise: the key facts one has to know are that any compact subset $K$ of $G(X)$ is contained in the closure $\overline{E}$ of some controlled set (which is itself compact and open), and that for any controlled set $E$ there is a constant $M>0$ (coming from bounded geometry) such that for any $f\in C_c(G(X))$ with support in $\overline{E}$, we have 
$$
\frac{1}{M}\|f\|_I\leq \sup_{x,y}|a^f_{xy}|\leq \|f\|_I.
$$
We leave the remaining details to the reader.
\end{proof}

Note that the isomorphism $C_c(G(X))\cong \C_u[X]$ from the proof above gives rise to a natural representation of $C_c(G(X))$ on $\ell^2(X)$ by matrix multiplication of the corresponding element of $\C_u[X]$.  Let $\mathcal{F}_{\ell^2(X)}$ be the family of representations of $C_c(G(X))$ consisting of this single representation.  Then we get an interesting example of property (T) with respect to $\mathcal{F}_{\ell^2(X)}$ coming from \emph{expanders} as in the following definition (see the book \cite{Lubotzky:1994tw} for background on expanders). 

\begin{definition}\label{exp def}
We will think of edges in a graph $X$ as two-element subsets of $X$; in particular, our graphs have no loops, no multiple edges, and are undirected.  Let $X=(X_n)_{n=1}^\infty$ be a sequence of finite, connected graphs.  We will abuse notation, and also write $X$ for the disjoint union $X=\bigsqcup_{n=1}^\infty X_n$.  Assume that there is an absolute bound on the degree of all vertices in $X$, and that the cardinality of $X_n$ tends to infinity as $n$ tends to infinity.  Let $\mathcal{E}$ be the coarse structure on $X$ generated be the set 
$$
\{(x,y)\in X\times X\mid \{x,y\} \text{ an edge}\}
$$ 
and the diagonal.  Then the coarse space $X$ is bounded geometry (due to the bound on vertex degrees) and monogenic.

For each $n$, let now $\Delta_n$ be the graph Laplacian on $\ell^2(X_n)$ defined by 
$$
\Delta_n:\delta_x\mapsto \sum_{\{y,x\} \text{ an edge}}\delta_x-\delta_y.
$$
It follows from the formula 
$$
\langle \xi,\Delta_n \xi\rangle=\sum_{\{x,y\}\text{ an edge} }|\xi(x)-\xi(y)|^2
$$
that $\Delta_n$ is a positive operator with kernel consisting exactly of the constant functions in $\ell^2(X_n)$ (this uses that $X_n$ is connected).  The sequence $X$ is an \emph{expander} if there exists a constant $c>0$ such that for all $n$ the spectrum of $\Delta_n$ is contained in $\{0\}\cup [c,\infty)$.
\end{definition}

\begin{proposition}\label{exp t}
Let $X$ be an expander.  Then the associated coarse groupoid $G(X)$ has property (T) with respect to the singleton family $\mathcal{F}_{\ell^2(X)}$ consisting of the natural representation on $\ell^2(X)$.
\end{proposition}

\begin{proof}
It is not difficult to check that a vector $\xi$ in $\ell^2(X)$ is constant for this representation of $C_c(G(X))$ if and only if it is constant as a function $X_n\to \C$ for each $n$.  Let $\Delta$ denote the operator on $\ell^2(X)$ that acts by $\Delta_n$ on each subspace $\ell^2(X_n)$.  If $\xi\in \ell^2(X)$ is in the orthogonal complement of the constant vectors, we must have that 
$$
\langle \xi,\Delta\xi\rangle \geq c\|\xi\|^2
$$
by the above comments on the spectrum and kernel of each $\Delta_n$.  On the other hand, a little combinatorics (compare \cite[Section 5]{Willett:2013cr}) shows that one can write 
$$
\Delta=\sum_{i=1}^n (v_iv_i^*-v_i)^*(v_iv_i^*-v_i)
$$ 
for some collection of partial isometries, each of which is represented by a $\{0,1\}$-valued function in $C_c(G(X))$ supported on a bisection.  We thus have that 
$$
\sum_{i=1}^n \|(v_iv_i^*-v_i)\xi\|^2\geq c\|\xi\|^2.
$$
As each $v_i$ is supported on a bisection, we have moreover that $v_iv_i^*=\Psi(v_i)$.  We must therefore have that for some $i$ 
$$
 \|(v_iv_i^*-v_i)\xi\|\geq \frac{\sqrt{c}}{n}\|\xi\|,
$$
which gives the desired result.
\end{proof}

\subsection{HLS groupoids and property $\tau$}

Our aim in this subsection is to discuss so-called HLS groupoids, and a connection to property $(\tau)$.  HLS groupoids are constructed from a discrete group and a collection of finite quotients; they were introduced by Higson, Lafforgue, and Skandalis in \cite[Section 2]{Higson:2002la} as part of their work on counterexamples to the Baum-Connes conjecture.  Property $(\tau)$ is a version of property (T) for groups that only sees information from representations that factor through finite quotients; see the book \cite{Lubotzky:1994tw} for background.   

The key ingredients for the construction of HLS groupoids are a discrete group, and an \emph{approximating sequence} $\mathcal{K}$ of subgroups: this means $\mathcal{K}$ is a nested sequence 
$$
K_1\geq K_2\geq \cdots 
$$
of finite index normal subgroups of $\Gamma$ such that the intersection $\bigcap K_n$ is the trivial group.  Given such a group and approximating sequence, let $\Gamma_n:=\Gamma/K_n$ be the corresponding quotient group for each $n$, and $q_n:\Gamma\to \Gamma_n$ the quotient map.  Define also $\Gamma_\infty=\Gamma$, and $q_\infty:\Gamma\to\Gamma_\infty$ to be the identity map.

\begin{definition}\label{hls gpd}
Let $\Gamma$ be a discrete group with a fixed approximating sequence $\mathcal{K}$ as above.  Let $\overline{\N}=\N\cup\{\infty\}$ be the one-point compactification of the natural numbers, equipped with the usual topology and order structure.  The associated \emph{HLS groupoid} has as underlying set 
$$
G_\mathcal{K}:=\bigsqcup_{n\in \overline{\N}} \{n\}\times \Gamma_n.
$$
It is equipped with the topology generated by the following sets: $\{(n,g)\}$ for $n\in \N$ and $g\in \Gamma_n$; and $\{(n,q_n(g))\in G_\mathcal{K}\mid n\in \overline{\N}, n\geq N\}$ as $N$ ranges over $\N$, and $g$ over $\Gamma$.  The base space is 
$$
G^{(0)}:=\{(n,g)\in G_\mathcal{K}\mid g\text{ is the identity $e$ of }\Gamma_n\},
$$
and the range and source maps are given by $r(n,g)=s(n,g)=(n,e)$.  Composition and inverses are defined using the group operations in each fibre $\{n\}\times \Gamma_n$.
\end{definition}

For an HLS groupoid $G_{\mathcal{K}}$ built as above, we call $\Gamma$ the \emph{parent group}.

In \cite[Lemma 2.4]{Spakula:2013ys}, it was proved that $G_{\mathcal K}$ is (topologically) amenable if and only if the parent group $\Gamma$ is amenable; thus amenability of $G_{\mathcal{K}}$ only sees the parent group and not the approximating sequence.  In this section, we will show a similar result for property (T): $G_\mathcal{K}$ has property (T) if and only if the parent group $\Gamma$ does.  More subtly, we will also give a result that takes the approximating sequence into account: $G_\mathcal{K}$ has property (T) with respect to the family of representations that extend to $C^*_r(G_{\mathcal{K}})$ if and only if $\Gamma$ has property $(\tau)$ with respect to the approximating sequence $\mathcal{K}$ (we recall the definition of property $(\tau)$ below).

For both results, we need a lemma relating $C_c(G_{\mathcal{K}})$ to the group algebra $\C[\Gamma]$.  

\begin{lemma}\label{hls group}
Let $G_{\mathcal{K}}$ be an HLS groupoid associated to the discrete group $\Gamma$ and approximating sequence $\mathcal{K}$.  Then restriction to the fibre at infinity defines a surjective $*$-homomorphism $\sigma:C_c(G_{\mathcal{K}})\to \C[\Gamma]$.  On the other hand, for each $g\in \Gamma$, set $\chi_g$ to be the characteristic function of the set 
$$
\{(n,q_n(g))\in G_{\mathcal{K}}\mid n\in \N\}.
$$
Then the map 
$$
\Gamma\to C_c(G_{\mathcal{K}}), \quad g\mapsto \chi_g
$$
extends to an injective $*$-homomorphism $\iota:\C[\Gamma]\to C_c(G_{\mathcal{K}})$. 
\end{lemma}

\begin{proof}
The proof consists of direct checks that we leave to the reader.  Note that injectivity of $\iota$ follows as $\iota$ is split by $\sigma$.
\end{proof}

We now get to the first of our main results.

\begin{proposition}\label{hls t}
Let $G_{\mathcal{K}}$ be an HLS groupoid with parent group $\Gamma$.  Then $G_{\mathcal{K}}$ has property (T) if and only if $\Gamma$ has property (T).
\end{proposition}

\begin{proof}
Assume first that $\Gamma$ has property (T), so there is a finite Kazhdan set $S$ with associated constant $c>0$.  Let $\pi$ be a representation of $C_c(G_{\mathcal{K}})$ on some Hilbert space $H$, and consider the representation $\pi\circ \iota$ of $\C[\Gamma]$.  It is straightforward to check that the invariant vectors for $\pi$ are the same as those for $\pi\circ \iota$.  From this, one sees that the set $K:=\{(n,q_n(g))\in G_{\mathcal{K}}\mid n\in \N, g\in S\}$ is a groupoid Kazhdan set: indeed, the function $f$ with support contained in $K$ required by the definition can always be taken to be one of the functions $\chi_g$ for some $g\in S$.  We leave the remaining details to the reader.

Conversely, say $G_{\mathcal{K}}$ has property (T), with associated Kazhdan set $K$.  Let $S=\{g\in \Gamma\mid (\infty,g)\in K\}$.  We claim that this $S$ is a group Kazhdan set for $\Gamma$.  Indeed, if $u$ is a unitary representation of $\Gamma$, denote also by $u$ the corresponding $*$-representation of $\C[\Gamma]$.  With $\sigma$ as in Lemma \ref{hls group}, the composition $u\circ \sigma$ is then a representation of $C_c(G_{\mathcal{K}})$.  It is straightforward to check that groupoid invariant vectors for $u\circ \sigma_\infty$ are the same thing as group invariant vectors for $u$, and from here that $K$ being a groupoid Kazhdan set implies that $S$ is a group Kazhdan set; we again leave the details to the reader.
\end{proof}

We now turn to property $(\tau)$.  We give a definition that is a little more general than necessary as it will be useful later.

\begin{definition}\label{iso def}
Let $\mathcal{U}$ be a collection of unitary representations of a discrete group $\Gamma$.  A subset $S$ of $\Gamma$ is a \emph{Kazhdan set for $\mathcal{U}$} if there exists $c>0$ such that if $(H,u)$ is a unitary representation of $\Gamma$ contained in $\mathcal{U}$ and such that 
$$
\|u_g\xi-\xi\|<c\|\xi\|
$$
for all $g\in S$, then there exists a non-zero invariant vector in $H$.

The group $\Gamma$ has \emph{property (T)} with respect to the collection $\mathcal{U}$ if it admits a finite Kazhdan set.
\end{definition}

\begin{example}\label{gp t ex}
A group $\Gamma$ has property (T) in the usual sense of Definition \ref{gpt} if and only if it has property (T) with respect to the family of all representations.  In particular, if $\Gamma$ has property (T), then it has property (T) with respect to any collection of representations.
\end{example}

\begin{definition}\label{tau def}
Let $\Gamma$ be a discrete group, and $\mathcal{K}$ an approximating sequence.  Let $\mathcal{U}_{\mathcal{K}}$ be the collection of unitary representations of $\Gamma$ that factor through one of the finite quotients $\Gamma_n$ for some $n\in \N$.  Then $\Gamma$ has \emph{property $(\tau)$ with respect to $\mathcal{K}$} if it has property (T) with respect to $\mathcal{U}_{\mathcal{K}}$.
\end{definition}

\begin{proposition}\label{hls tau}
Let $G_{\mathcal{K}}$ be an HLS groupoid with parent group $\Gamma$.  Let $\mathcal{R}$ be the collection of representations of $C_c(G_{\mathcal{K}})$ that extend to the regular representation $C^*_r(G_{\mathcal{K}})$.  Then $G_{\mathcal{K}}$ has property (T) with respect to $\mathcal{R}$ if and only if $\Gamma$ has property $(\tau)$ with respect to $\mathcal{K}$.
\end{proposition}

\begin{proof}
Let $C^*_\mathcal{K}(\Gamma)$ denote the completion of the group algebra $\C[\Gamma]$ for the norm 
$$
\|a\|:=\sup_{u\in \mathcal{U}_{\mathcal{K}}}\|u(a)\|
$$
Note that $\Gamma$ has property (T) with respect to the collection $\mathcal{U}_{\mathcal{K}}$ if and only it has property (T) with respect to the collection of all representations of $\Gamma$ that extend to $C^*_{\mathcal{K}}(\Gamma)$.  

Having made the above definition and observation, the proof of the proposition is then essentially the same as that of Proposition \ref{hls t}, once we have noted also that: the map $\iota:\C[\Gamma]\to C_c(G_{\mathcal{K}})$ of Lemma \ref{hls group} extends to an injective $*$-homomorphism $C^*_{\mathcal{K}}(\Gamma)\to C^*_r(G_{\mathcal{K}})$; and that the map $\sigma:C_c(G_{\mathcal{K}})\to \C[\Gamma]$ of Lemma \ref{hls group} extends to a surjective $*$-homomorphism $C^*_r(G_{\mathcal{K}})\to C^*_{\mathcal{K}}(\Gamma)$ (compare the proof of \cite[Lemma 2.7]{Spakula:2013ys}).  We leave the remaining details to the reader. 
\end{proof}

\subsection{Group actions}

Let $\Gamma$ be a discrete group acting on a compact space $X$.  Our goal in this section is to characterise property (T) for the associated \emph{transformation groupoid} $X\rtimes \Gamma$.  We start with the definitions.

Recall then that the \emph{transformation groupoid} $G:=X\rtimes \Gamma$ associated to such an action is defined as a set to be.
$$
G:=\{(gx,g,x)\in X\times \Gamma\times X\mid g\in \Gamma, x\in X\}.
$$
It is equipped with the subspace topology it inherits from $X\times \Gamma\times X$.  The unit space is $G^{(0)}=\{(x,e,x)\mid x\in X\}$ (where $e$ is the trivial element in $\Gamma$), which we identify with $X$ in the obvious way.  The range and source maps $r,s:G\to X$ are given by 
$$
r:(gx,g,x)\mapsto gx,\quad s:(gx,g,x)\mapsto x
$$
respectively, and the composition and inverse by 
$$
(ghx,g,hx)(hx,h,x)=(ghx,gh,x) \quad \text{and}\quad (gx,g,x)^{-1}=(x,g^{-1},gx).
$$

The following lemma is well known; we provide a sketch proof for the reader's convenience, and as we need to establish notation.  In order to state it, for $g\in \Gamma$, let us write $G_g:=\{(gx,g,x)\in G\mid x\in X\}$ for the `slice' of $G$ corresponding to $g$, and let us write $\alpha_g$ for the $*$-automorphism of $C_c(X)$ defined by $\alpha_g(f):=f(g^{-1}x)$.

\begin{lemma}\label{cov pair}
Let $\pi:C_c(G)\to \mathcal{B}(H)$ be a unital representation of $C_c(G)$.  Then there exist unique representations $\pi^X$ and $\pi^\Gamma$ of $C_c(X)$ and $\Gamma$ respectively on $H$ that satisfy the \emph{covariance relation} 
$$
\pi^\Gamma_g \pi^X(f)(\pi^\Gamma_g)^*=\pi^X(\alpha_g(f))
$$
and such that for all $f\in C_c(G)$
\begin{equation}\label{rep form}
\pi(f)=\sum_{g\in \Gamma} \pi^X(\Psi(f|_{G_g}))\pi^\Gamma_g.
\end{equation}
Conversely, any pair of representations $(\pi^X,\pi^\Gamma)$ of $C_c(X)$ and $\Gamma$ on some $H$ that satisfy the covariance relation uniquely determines a nondegenerate representation of $C_c(G)$ via the formula in line \eqref{rep form}.
\end{lemma}

\begin{proof}
Starting with a representation $\pi$ of $C_c(G)$, define $\pi^X$ to be the restriction of $\pi$ to $C(X)\subseteq C_c(G)$ (as usual, we identify $X$ with $G^{(0)}$ here).  For $g\in \Gamma$, define $u_{g}:G\to [0,1]$ to be the function 
$$
u_{g}(hx,h,x):=\left\{\begin{array}{ll} 1 & h=g \\ 0 & \text{otherwise}\end{array}\right..
$$
We leave the direct checks that (a) $g\mapsto u_g$ defines a unitary representation of $\Gamma$, (b) of the covariance relation, and (c) of the equation in line \eqref{rep form} to the reader.

The converse direction is straightforward: given a covariant pair $(\pi^X,\pi^\Gamma)$, define $\pi$ by the formula in line \eqref{rep}, and use the covariance relation to show that this does define a representation of $C_c(G)$; we leave the direct computations involved to the reader.
\end{proof}

The next lemma again consists of direct algebraic computations; this time we leave all the details to the reader.

\begin{lemma}\label{fix vect}
Let $\pi$ be a nondegenerate representation of $C_c(G)$ on $H$, and let $(\pi^X,\pi^\Gamma)$ be the corresponding covariant pair from Lemma \ref{cov pair}.  Then a vector $\xi$ in $H$ is fixed by $C_c(G)$ if and only if it is invariant for $\Gamma$ in the sense that $\pi^\Gamma_g \xi=\xi$ for all $g\in \Gamma$.  \qed
\end{lemma}

Going back to actions, the following definition is natural.

\begin{definition}\label{ux def}
Let $\mathcal{U}_X$ be the collection of all representations $u$ of $\Gamma$ such that there exists a unital representation $\pi$ of $C(X)$ with $(\pi,u)$ covariant.
\end{definition}

\begin{proposition}\label{trans t}
Let $\Gamma$ be a discrete group acting on a compact space $X$, and let $G=X\rtimes \Gamma$ be the associated transformation groupoid.  Then the following are equivalent:
\begin{enumerate}[(i)]
\item $G$ has property (T);
\item $\Gamma$ has property (T) with respect to the collection $\mathcal{U}_X$ in the sense of Definition \ref{iso def}.
\end{enumerate} 
\end{proposition} 

\begin{proof}
Assume $G$ has property (T), and let $(K,c)$ be a Kazhdan pair for $G$ with $K$ compact.  Let $u$ be a representation in $\mathcal{U}_X$, so $u$ is part of some covariant pair $(\pi^X,u)$.  Let $\pi$ be the corresponding representation of $C_c(G)$ as in Lemma \ref{cov pair}.   Using Lemma \ref{fix vect}, the orthogonal complement of the $u$ fixed vectors exactly corresponds to $H_\pi$.  Let $\xi$ be a unit vector in $H_\pi$, and let $f\in C_c(G)$ be supported in $K$, such that $\|f\|_I\leq 1$, and with the property that $\|\pi(f)\xi-\pi(\Psi(f))\xi\|\geq c$.  As $K$ is compact, we have that $K$ is contained in $\{(gx,g,x)\in G\mid g\in S\}$ for some finite subset $S$ of $G$.  We may write $f$ as a finite sum $f=\sum_{g\in S}f|_{G_g}$; note that $\|f|_{G_g}\|_I\leq 1$ for each $g\in S$.  There must then exist some $g\in S$ such that $\|\pi(f|_{G_g})\xi-\pi(\Psi(f|_{G_g}))\xi\|\geq c/|S|$.  Note that 
$$
\pi(f|_{G_g})=\pi(\Psi(f|_{G_g}))u_g,
$$
whence we now have that for some $g\in S$
$$
c/|S|\leq \|\pi(\Psi(f|_{G_g}))u_g\xi-\pi(\Psi(f|_{G_g}))\xi\|\leq \|u_g\xi-\xi\|,
$$
giving us that $\Gamma$ has property (T) with respect to $\mathcal{U}_X$.

For the converse direction, assume that $\Gamma$ has property (T) with respect to $\mathcal{U}_X$, and let $(S,c)$ be a Kazhdan pair in the usual sense.  Let $K:=\{(gx,g,x)\in G\mid g\in S\}$, which is compact.  We claim that $(K,c)$ is a Kazhdan pair for $G$, thus showing that $G$ has property (T).  Indeed, let $\xi\in H_\pi$ be a unit vector for some representation $(\pi,H)$ with $(\pi^X,\pi^\Gamma)$ the corresponding covariant pair as in Lemma \ref{cov pair}.  Then analogously to the discussion above there exists $g\in S$ such that $\|\pi^\Gamma_g\xi-\xi\|>c$.  Let $f\in C_c(G)$ be the characteristic function of the slice $G_g:=\{(gx,g,x)\mid x\in X\}$.  Then $f$ is supported in $K$, satisfies $\|f\|_I\leq 1$, and the above says that $\|\pi(f)\xi-\pi(\Psi(f))\xi\|>c$, so we are done.
\end{proof}

\begin{corollary}\label{t inv meas}
Let $\Gamma$ be a discrete group acting on a compact space $X$, and let $G=X\rtimes \Gamma$ be the associated transformation groupoid.  Assume moreover that $X$ admits an invariant probability measure.  Then $G=\Gamma\rtimes X$ has property (T) if and only if $G$ has property (T).
\end{corollary}

\begin{proof} 
If $\Gamma$ has property (T), then $G$ always has property (T) by Example \ref{gp t ex} and Proposition \ref{trans t}.  Conversely, the multiplication representation $\pi^\mu$ and permutation representation $u^\mu$ of $C_c(X)$ and $\Gamma$ respectively on $L^2(X,\mu)$ fit together to make a covariant pair.  Moreover, $u_\mu$ contains the trivial representation as a subrepresentation.  It follows that if $(H,u)$ is any unitary representation of $G$, then $(\pi^\mu\otimes 1_H,u^\mu\otimes u)$ is a covariant pair such that the $\Gamma$ part $u^\mu\otimes u$ contains $u$ as a subrepresentation.  As $u$ was arbitrary, it follows that property (T) with respect to $\mathcal{U}_X$ is the same as property (T) with respect to the collection of all unitary representations, which is just property (T).
\end{proof}

\begin{example}\label{spec gap rem}
Let $\Gamma$ be a discrete group.  By definition, a compact space $X$ with an action of $\Gamma$ and a quasi-invariant measure $\mu$ has \emph{spectral gap} if $\Gamma$ has property (T) with respect to the collection of representations consisting of just the Koopman representation on $L^2(X,\mu)$.  From Proposition \ref{trans t}, it follows that if $G=X\rtimes \Gamma$ has property (T) and $\mu$ is a quasi-invariant measure on $X$, then the action of $\Gamma$ on $(X,\mu)$ will have spectral gap. 
\end{example}

\section{Connections with other versions of property (T)}\label{meas sec}

In this section, we explore relationships with other versions of property (T): first other topological notions, then the measure-theoretic definition of Zimmer and Anantharaman-Delaroche.

\subsection{Other topological definitions of property (T)}

There are two other versions of topological property (T) for groupoids that either seem reasonable, or have appeared more-or-less explicitly in the literature.  In this subsection, we look at these, and (at least partially) determine the relationship to our notion.    As usual, throughout this section, `groupoid' means locally compact, Hausdorff, \'{e}tale groupoid with compact base space: see Convention \ref{gpd conv}.

The first possible variant of property (T) is as follows, and is a natural variant of our notion from Definition \ref{basic t}.

\begin{definition}\label{basic t1}
Let $G$ be a groupoid with compact base space.  A subset $K$ of $G$ is a \emph{Kazhdan$_1$ set} if there exists $c>0$ such that for any representation $(H,\pi)$ of $C_c(G)$ which does not have invariant vectors, and any $\xi\in H$, there exists $f\in C_c(G)$ with support in $K$ and $\|f\|_I\leq 1$ such that $\|f\xi-\Psi(f)\xi\|\geq c\|\xi\|$.

The groupoid $G$ has \emph{(topological) property (T$_1$)} if it admits a compact Kazhdan$_1$ set.
\end{definition}

For groups, it follows from the fact that the invariant vectors $H^\pi$ form a subrepresentation of any given representation $(H,\pi)$ that property (T) is equivalent to property (T$_1$).  Clearly we also have that property (T) implies (T$_1$) in general; the converse, however, is false as we will see in a moment.  For certain purposes, property (T$_1$) may be more natural than property (T), partly as it deals with genuine representations rather than subspaces of representations.  However, for our main applications on the construction of Kazhdan projections, property (T) is the more useful version.

Here is an example showing that property (T) is strictly stronger than property (T$_1$).

\begin{example}\label{ t1 no t} 
Let $\overline{\N}$ be the one-point compactification of the natural numbers considered as a trivial groupoid, and let $P$ be the pair groupoid on $\{0,1\}$.  Let $\overline \N \times P$ be the associated product groupoid: its base space is $\overline \N \times P^{(0)}$ with the obvious structure maps. Let $G$ be the subgroupoid of $\overline \N \times P$ defined by 
$$ G := (\N \times P) \cup (\{\infty \}\times P^{(0)}),$$
i.e.\ $\gamma= (n,g)\in \overline{\N}\times P^{(0)}$ is in $G$ if and only if $n\neq \infty$ or $g\in P^{(0)}$.   As $G$ is an open subgroupoid of $\overline{\N}\times P$, it is \'{e}tale.  We will denote by $G_n$ the subgroupoid sitting over $n$, i.e.\ the restriction $r^{-1}(\{n\}\times P^{(0)})$.  Note that $C_c(G_n)\cong M_2(\C)$ for each $n\in \N$; we will fix such an isomorphism that takes the characteristic functions of the points $0$ and $1$ in the unit space of $G_n$ to the two diagonal projections in $M_2(\C)$.  We have the following exact sequence
\begin{equation}\label{tt1ses}
0 \rightarrow \bigoplus_{n\in \N} C_c(G_n)\rightarrow C^*(G) \rightarrow \C^2 \rightarrow 0.
\end{equation}

We will show that $G$ has $(T_1)$ but not $(T)$. 
\begin{enumerate}
\item We first claim that $G$ has $(T_1)$.  Let $\phi : C_c(G)\rightarrow \mathcal{B}(H)$ be a representation.  If there is some $n$ such that $\phi$ is non-zero when restricted to $C_c(G_n)$, we claim that $\phi$ has invariant vectors.  Indeed, if $p\in C_c(G_n)$ is the projection corresponding to the matrix $\frac{1}{2}\begin{pmatrix} 1 & 1 \\ 1 & 1 \end{pmatrix}\in M_2(\C)$, then $\phi(p)$ is non-zero as $C_c(G_n)$ is simple.  Any non-zero element in the image of $\phi(p)$ is invariant.

On the other hand, a representation $\phi$ that is zero on all the subalgebras $C_c(G_n)$ factors through the quotient in line \eqref{tt1ses} as $\tilde \phi : \C^2 \rightarrow \mathcal{B}(H)$.  All vectors are invariant for such a representation. This shows that any representation $C^*(G)\rightarrow \mathcal{B}(H)$ has non-trivial invariant vectors, hence $(T_1)$ holds for vacuous reasons.
\item We now show that $G$ does not have $(T)$.  For the sake of contradiction, let us suppose a Kazhdan set $K$ exists. By compactness of $K$, there exists $N\in\N$ such that
$$ K \subseteq (\{0 , ... ,N\}\times P) \cup G^{(0)}.$$
Let now $\phi:C_c(G)\to \C^2$ be the representation one gets by composing the natural restriction map $C_c(G)\to C_c(G_{N+1})$ with the canonical representation of $C_c(G_{N+1})\cong M_2(\C)$ on $H:=\C^2$.  We then have that the vector $(-1,1)^T\in H$ is non-zero and in the subspace $H_\phi$ that is orthogonal to the constant vectors.  However, for any $f\in C_c(G)$ supported in $K$, the restriction of $f$ to $G_{N+1}$ is supported in $G_{N+1}^{(0)}$.  Hence $\phi(f)=\phi(\Psi(f))$, so this contradicts property $(T)$.
\end{enumerate}

\end{example}

The second definition of property (T) that we look at is also very natural.  This has appeared in the literature before for group actions in a slightly different but equivalent form: see \cite[page 441]{Dong:2012aa}.  We are not aware of any study of the general groupoid property in the literature before.

We need a standard preliminary definition: see for example \cite[Definition 5.6.15]{Brown:2008qy}.  

\begin{definition}\label{pt fun def}
Let $G$ be a groupoid.  A function $\phi:  G \rightarrow \C$ is \emph{positive type} if:
\begin{enumerate}[(i)]
\item $\phi(x)=1$ for all $x\in G^{(0)}$;
\item $\phi$ is symmetric, i.e. $\phi(g^{-1}) = \overline{\phi(g)}$ for every $g\in G$;
\item for every finite tuple $g_1, ... , g_n$ in $G$ with the same range and every tuple $z_1, ... , z_n$ of complex numbers, 
\[ \sum_{i,j=1}^n  \overline{z_i} z_j \phi(g_i^{-1} g_j ) \geq 0.\] 
\end{enumerate}
\end{definition} 

\begin{definition}\label{t2 def}
A groupoid $G$ has \emph{(topological) property (T$_2$)} if whenever $(\phi_i:G\to \C)_{i\in I}$ is a net of positive type functions in $C_c(G)$ that converges uniformly on compact sets to the constant function one, then $(\phi_i)$ converges uniformly to the constant function one. 
\end{definition}

The above definition is well-known to be equivalent to property (T) in the group case: this follows for example from \cite[Lemma 2]{Akemann:1981dz} combined with \cite[Theorem 13.5.2]{Dixmier:1977vl}.  It is moreover a very natural definition, and maybe of a more `topological' nature than ours: indeed, ours has some measure-theoretic flavour coming from the connections of invariant vectors to invariant measures, and also from the connection to representation theory.

The following lemma combined with Proposition \ref{trans t} shows that in the case of group actions, property (T$_2$) is strictly stronger than our property (T).

\begin{lemma}\label{t2 trans}
Let $\Gamma$ be a discrete group acting on a compact space $X$ by homeomorphisms, and let $X\rtimes \Gamma$ be the associated transformation groupoid.  If $X\rtimes \Gamma$ has property (T$_2)$, then $\Gamma$ has property (T).
\end{lemma}

\begin{proof}
Assume $G:=X\rtimes \Gamma$ has property (T$_2$), and let $(\phi_i:\Gamma \to \C)$ be a net of positive type functions converging uniformly on compact sets (i.e.\ pointwise, as $\Gamma$ is discrete) to the constant function one; to see that $\Gamma$ has (T) it suffices to prove that $(\phi_i)$ converges uniformly to one.  To see this, for each $i$ let $\widetilde{\phi_i}:G\to \C$ be the pullback defined by 
$$
\widetilde{\phi_i}(gx,g,x):=\phi_i(g).
$$
Then direct checks show that each $\widetilde{\phi_i}$ is positive type, and that the net $(\widetilde{\phi_i})$ converges uniformly to one on compact subsets of $G$; hence by property (T$_2$) it converges uniformly to one.  It follows that the original net $(\phi_i)$ also converges uniformly to one, so we are done.
\end{proof}

Using the discussion in \cite[Section 11.4.3]{Roe:2003rw}, one also has the following result, showing that property (T$_2$) is essentially trivial for coarse groupoids.

\begin{lemma}\label{t2 cg}
Let $X$ be a bounded geometry metric space.  Then the coarse groupoid $G(X)$ has property (T$_2$) if and only if $X$ is bounded. \qed
\end{lemma}

Hence for coarse groupoids, property (T$_2$) is also strictly stronger than our property (T) by Proposition \ref{t geot}.  It is plausible from these examples that (T$_2$) implies (T) in general, but we were unable to show this.

\subsection{Measured property (T)}

In \cite{Anantharaman-Delaroche:2005pb}, Anantharaman-Delaroche defined a notion of property (T) for a measured groupoid, building on earlier work of Zimmer \cite{Zimmer:2981ds} in the case of a measured equivalence relation.  Our aim in this subsection is to discuss the relationship of this measure-theoretic notion to our topological notion: in particular (Theorem \ref{t to t} below), we show that the topological notion implies the measure-theoretic one for a large class of measures

Throughout this subsection $G$ will be a groupoid (as usual, locally compact, Hausdorff, \'{e}tale, and with compact unit space).  As we are interested in measure theory, we will assume that $G$ is second countable to avoid measure-theoretic pathologies.   We assume moreover that the base space $G^{(0)}$ is equipped with an invariant probability measure $\mu$.  Associated to this measure $\mu$ is the measure $r^*\mu$ on $G$ defined as a functional on $C_c(G)$ by the formula
$$
r^*\mu:f\mapsto \int_{G^{(0)}}\sum_{g\in G^x}f(g)d\mu(x).
$$
We equip $G$ with the Borel structure induced by the topology, and with the measure class $C$ of $r^*\mu$.  When we say `almost everywhere' below, we mean with respect to $\mu$ when the ambient space is $G^{(0)}$, and with respect to $C$ when the ambient space is $G$.  The pair $(G,C)$ is a measured groupoid in the sense of \cite[Definition 2.7]{Anantharaman-Delaroche:2005pb}.  As $C$ is determined by $\mu$, we will generally write $(G,\mu)$ for this measured groupoid.

We want to compare property (T) for $(G,\mu)$ in the sense of \cite[Section 4]{Anantharaman-Delaroche:2005pb} with our notion of property (T) for $G$.  To avoid confusion, let us call the former property \emph{measured property (T)} for $(G,\mu)$, and the latter property \emph{topological property (T)} for $G$.  

We first recall the definitions necessary to make sense of measured property (T).  The following is \cite[Definition 3.1]{Anantharaman-Delaroche:2005pb}.

\begin{definition}\label{rep}
A \emph{representation} of $G$ consists of the following data:
\begin{enumerate}[(i)]
\item a Hilbert bundle $H=(H_x)_{x\in G^{(0)}}$ over $G^{(0)}$ in the sense of \cite[Definition 2.2]{Anantharaman-Delaroche:2005pb};
\item the associated Borel groupoid $\text{Iso}(G^{(0)}*H)$ consisting of triples $(x,V,y)$ where $V:H_y\to H_x$ is a unitary isomorphism \cite[Section 3.1]{Anantharaman-Delaroche:2005pb};
\item a Borel homomorphism $\pi:G\to \text{Iso}(G^{(0)}*H)$ sending each unit $x\in G^{(0)}$ to the corresponding unit $(x,\text{Id}_{H_x},x)$ of $\text{Iso}(G^{(0)}*H)$.
\end{enumerate}
\end{definition}

We will write representations of $G$ in the sense above as pairs $(H,L)$.  We will abuse notation by writing $\pi_g:H_{s(g)}\to H_{r(g)}$ for the unitary $V$ such that $\pi_g=(r(g),V,s(g))$.  

The next definitions are from \cite[Sections 2.1 and 4.1]{Anantharaman-Delaroche:2005pb}.

\begin{definition}\label{sec}
Let $H$ be a Hilbert bundle over $G^{(0)}$ in the sense of \cite[Definition 2.2]{Anantharaman-Delaroche:2005pb}.  The space $S(G^{(0)},\mu,H)$ consists of all Borel sections 
$$
\xi:G^{(0)}\to H,\quad x\mapsto \xi_x
$$ 
(where `section' means that $\xi(x)\in H_x$), modulo almost everywhere equality, and equipped with the topology defined by the equivalent conditions from \cite[Proposition 2.3]{Anantharaman-Delaroche:2005pb}.   An element $\xi$ of $S(G^{(0)},\mu,H)$ is a \emph{unit section} if $\|\xi_x\|_{H_x}=1$ for almost all $x\in G^{(0)}$ (see \cite[Section 4.1]{Anantharaman-Delaroche:2005pb}).  
\end{definition}

The next definitions are from \cite[Definition 4.2]{Anantharaman-Delaroche:2005pb}.

\begin{definition}\label{inv}
Let $(H,\pi)$ be a representation of $G$.  
\begin{enumerate}[(i)]
\item A section $\xi$ in $S(G^{(0)},\mu,H)$ is \emph{invariant} if 
$$
\xi_{r(g)}=\pi_g\xi_{s(g)}\quad \text{in}\quad H_{r(g)}
$$
for almost every $g\in G$.
\item The representation $(H,\pi)$ \emph{almost contains unit invariant sections} if there is a sequence of unit sections $(\xi^n)$ such that 
$$
\|\xi^n_{r(g)}-\pi_g\xi^n_{s(g)}\|_{H_{r(g)}} \to 0
$$
for almost every $g\in G$.
\end{enumerate}
\end{definition}

Finally, we get to the definition of measured property (T) for our measured groupoid.  The following is \cite[Definition 4.3]{Anantharaman-Delaroche:2005pb}

\begin{definition}\label{meas t}\
Let $G$ be a groupoid (locally compact, Hausdorff, \'{e}tale, second countable, with compact base space) equipped with an invariant probability measure $\mu$ on $G^{(0)}$.  The measured groupoid $(G,\mu)$ has \emph{measured property (T)} if whenever a representation $(H,L)$ almost contains unit invariant sections, it actually contains a unit invariant section.
\end{definition}

\begin{remark}\label{ad gen}
Anantharaman-Delaroche's definition of measured property (T) applies to a more general class of measured groupoids than ours.  For example, Anantharaman-Delaroche does not assume the presence of an underlying topology, and allows quasi-invariant measures on the base space.  There seems to be no obvious connection between our definition and that of Anantharaman-Delaroche in the case of a quasi-invariant probability measure: see Lemma \ref{inv lem} and the following comments at the end of this section.
\end{remark}

Here is the main result of this section.

\begin{theorem}\label{t to t}
Let $G$ be a groupoid with topological property (T).  Then for every ergodic invariant probability measure $\mu$ on $G^{(0)}$, the measured groupoid $(G,\mu)$ has measured property (T).
\end{theorem}

\begin{proof}
Assume for contradiction that $\mu$ is an invariant ergodic measure on $G^{(0)}$ and $(H,\pi)$ a representation of $G$ that almost has unit invariant sections, but no invariant section.  Let $H_\mu$ be the Hilbert space completion of the collection of all bounded elements of $S(G^{(0)},\mu,H)$, equipped with the inner product 
$$
\langle \xi,\eta\rangle_{H_\mu}:=\int_{G^{(0)}} \langle \xi_x,\eta_x\rangle_{H_x}d\mu(x).
$$
As described in \cite[Section 2.3.3]{Renault:2009zr}, $(H,\pi)$ integrates to a $*$-representation 
$$
\pi:C_c(G)\to H_\mu
$$
with the property that for all $\xi,\eta\in H_\mu$,
$$
\langle \xi,\pi(f)\eta\rangle =\int_{G^{(0)}} \sum_{g\in G^x}f(g)\langle \xi_x,\pi_g\eta_{s(g)}\rangle_{H_x}d\mu(x).
$$

We claim first that the representation $(H_\mu,\pi)$ of $C_c(G)$ contains no non-zero constant vectors.  Assume for contradiction that $\xi\in H_\mu$ is a constant unit vector, so that 
\begin{equation}\label{const}
\pi(\Psi(f))\xi=\pi(f)\xi
\end{equation}
for all $f\in C_c(G)$.  Writing out what this means,
$$
(\pi(f)\xi)(x)=\sum_{g\in G^x}f(g)\pi_g\xi_{s(g)} \quad \text{and}\quad (\pi(\Psi(f))\xi)_x=\sum_{g\in G^x}f(g)\xi_x
$$
and so line \eqref{const} above says that 
$$
\sum_{g\in G^x}f(g)\pi_g\xi_{s(g)} =\sum_{g\in G^x}f(g)\xi_x
$$
for every $f\in C_c(G)$, and almost every $x\in G^{(0)}$.  As this holds for all $f\in C_c(G)$, considering functions $f$ that are supported on bisections (and using second countability) shows that this is impossible unless $\pi_g\xi_{s(g)}=\xi_{r(g)}$ for almost every $g\in G$.   This implies that the function 
$$
G^{(0)}\to \R, \quad x\mapsto \|\xi_x\|_{H_x}
$$ 
is invariant under the action of $G$ on $G^{(0)}$, and thus by ergodicity, it is constant almost everywhere.  As $\mu$ is a probability measure and as $\|\xi\|_{H_\mu}=1$, this forces $\|\xi_x\|=1$ for almost every $x\in G^{(0)}$.  At this point, we have that $\xi$ is a unit invariant section for $(H,\pi)$, which is the desired contradiction. 

Now, let $(\xi^n)$ be a sequence as in the definition of almost containing unit invariant sections, so that 
\begin{equation}\label{alm inv sec}
\|\xi^n_{r(g)}-\pi_g\xi^n_{s(g)}\|^2_{H_{r(g)}} \to 0
\end{equation}
for almost every $g\in G$.  From topological property (T) there exists a compact subset $K$ of $G$ and $c>0$ such that for each $\xi^n$ there exists $f_n\in C_c(G)$ supported in $K$ and with $\|f_n\|_I\leq 1$ such that 
$$
\|\pi(f_n)\xi^n-\pi(\Psi(f_n))\xi^n\|^2_{H_\mu}\geq c.
$$
Writing out what this means,
$$
\int_{G^{(0)}}\Big\|\sum_{g\in G^x}f_n(g)\pi_g\xi^n_{s(g)}-\sum_{x\in G^x}f_n(g)\xi^n_x\Big\|_{H_x}^2d\mu(x)\geq c.
$$
Using that $\|f_n\|_I\leq 1$ and that each $f_n$ is supported in $K$ we thus get
\begin{align}\label{meas t inqs}
c & \leq \int_{G^{(0)}}\Big\|\sum_{g\in G^x}f_n(g)\big(\pi_g\xi^n_{s(g)}-\xi^n_x\big)\Big\|_{H_x}^2d\mu(x) \nonumber \\
& \leq \int_{G^{(0)}}\Big(\sum_{g\in G^x}|f(g)|\|\pi_g\xi^n_{s(g)}-\xi^n_x\|_{H_x}\Big)^2d\mu(x) \nonumber \\ 
& \leq \int_{G^{(0)}}\sup_{g\in K\cap G^x}\|\pi_g\xi^n_{s(g)}-\xi^n_x\|_{H_x}^2d\mu(x).
\end{align}
Now, as $K\cap G^x$ is finite for all $x\in G^{(0)}$, line \eqref{alm inv sec} gives that the integrand above tends to zero pointwise almost everywhere.  As each $\xi^n$ is a unit section, the integrand is moreover bounded above by four; as $\mu$ is a probability measure we may thus apply the dominated convergence theorem to get that the final integral in line \eqref{meas t inqs} tends to zero as $n$ tends to infinity.  As it is bounded below by $c$ for all $n$, this gives the required contradiction. 
\end{proof}

To conclude this section, we make some comments about the relationship of our definition to that of Anantharaman-Delaroche when one only has a quasi-invariant measure on the base space.  The essential point is that the notions of constant vectors one gets in that case are different.

Recall then that if $G$ is a groupoid and $\mu$ is a quasi-invariant measure on $G^{(0)}$ then there is an associated modular function $D:G\to (0,\infty)$ defined by $D=d(r^*\mu) / d(s^*\mu)$.   If moreover $(H,\pi)$ is a representation of $G$ in the sense of Definition \ref{rep} above, then associated to the triple $(H,\mu,\pi)$ we may form the Hilbert space 
$$
H_\mu:=L^2(G^{(0)},\{H_x\},\mu)
$$
of $L^2$-sections of the family $\{H_x\}$ with respect to the measure $\mu$.  Moreover, there is a representation of $C_c(G)$ on $H_\mu$ uniquely determined by the condition
$$
\langle \xi,\pi(f)\eta\rangle=\int_{G^{(0)}} \sum_{g\in G^x}f(g)D^{-1/2}(g)\langle \xi_x,\pi_g\eta_{s(g)}\rangle_{H_x}d\mu(x)
$$
for all $\xi,\eta\in H_\mu$ and $f\in C_c(G)$.  The representation $(H_\mu,\pi)$ of $C_c(G)$ is called the \emph{integrated form} of the triple $(H,\mu,\pi)$.  Conversely, \emph{Renault's disintegration theorem} \cite[Theorem 2.3.15]{Renault:1980fk} says that when $G$ is second countable, any representation $(H,\pi)$ of $C_c(G)$ arises like this.

We leave the proof of the following lemma to the reader.

\begin{lemma}\label{inv lem}
Let $(H_\mu,\pi)$ be the integrated form of the representation $(H,\mu,\pi)$ of a second countable groupoid $G$.  Let $D$ be the modular function associated to $\mu$.  Then a vector $\xi\in H_\mu$ is constant in the sense of Definition \ref{con vec def} if and only if 
$$
\xi_{r(g)}=D^{-1/2}(g)\pi_g\xi_{s(g)}
$$
for almost all $g\in G$, where `almost all' is meant with respect to the measure $r^*\mu$. \qed
\end{lemma}

On the other hand, Anantharaman-Delaroche uses the definition of constant from Definition \ref{inv} above, that $\xi_{r(g)}=\pi_g\xi_{s(g)}$ for almost every $g\in G$, also in the case of a quasi-invariant measure of $G^{(0)}$.  Thus in the case when $\mu$ is only quasi-invariant, it seems unreasonable to expect much connection between the notions of Anantharaman-Delaroche (and also of Zimmer) and ours.

\section{Connections with a-T-menability}\label{atmen sec}

As usual in this section, groupoids are always locally compact, Hausdorff, \'{e}tale, and have compact base space: see Convention \ref{gpd conv}.

The property of a-T-menability for groupoids was introduced by Tu \cite[Section 3]{Tu:1999bq} as part of his work on the Baum-Connes conjecture.  Just as for groups, a-T-menability for groupoids is a generalisation of amenability that admits several useful characterisations.  Moreover, just as for groups, all amenable groupoids are a-T-menable.  

For groups, the name a-T-menability (due to Gromov) came about as this condition is like amenability, and incompatible with property (T): indeed a discrete group is a-T-menable and has property (T) if and only if it is finite.  Our goal in this section is to show that property (T) for a groupoid is also incompatible with a-T-menability in many cases.  

Here is a sample result that we can deduce from our main theorem.  To state it, recall that a groupoid is \emph{minimal} if for every $x\in G^{(0)}$, the \emph{orbit} $Gx$ defined by $Gx:=s(G^x)$ is dense in $G^{(0)}$.  As usual, we assume throughout the section that all our groupoids are locally compact, Hausdorff, \'{e}tale, and have compact base space.

\begin{theorem}\label{atmen t min}
Let $G$ be a minimal groupoid with property (T), that is a-T-menable, and such that $G^{(0)}$ admits an invariant probability measure.  Then $G$ is finite.
\end{theorem}

Note that this result generalises the above-mentioned incompatibility of a-T-menability and property (T) in the group case.  See also \cite[Proposition 4.7]{Anantharaman-Delaroche:2005pb} for an analogous result in the measured context, where the minimality assumption is replaced by the related measure-theoretic assumption of ergodicity.  

In order to get to our main result, we need some definitions.  We start by recalling some definitions from Tu's work \cite[Section 3]{Tu:1999bq}.

\begin{definition}\label{neg type def}
Let $G$ be a groupoid.  A function $F:  G \rightarrow [0,\infty)$ is of \emph{negative type} if:
\begin{enumerate}[(i)]
\item $F(x)=0$ for all $x\in G^{(0)}$;
\item $F$ is symmetric, i.e. $F(g^{-1}) = F(g)$ for every $g\in G$;
\item for every finite tuple $g_1, ... , g_n$ in $G$ with the same range, and every tuple $a_1, ... , a_n$ of real numbers such that $\sum_j a_j=0$, 
\[ \sum_{i,j=1}^n a_i a_j F(g_i^{-1} g_j ) \leq 0.\] 
\end{enumerate}
\end{definition} 

\begin{definition}\label{atmen def}
A groupoid $G$ is \emph{a-T-menable} if there exists a continuous, proper\footnote{If we do not assume that $G^{(0)}$ is compact, `proper' should be replaced with `locally proper': a function $F:G\to [0,\infty)$ is \emph{locally proper} if for any compact subset $K$ of $G^{(0)}$, the restriction of $F$ to $G_K^K$ is proper in the usual sense.}, negative type function $F:G\to [0,\infty)$.
\end{definition}

Tu shows several useful facts about the class of a-T-menable groupoids in \cite[Section 3]{Tu:1999bq}: perhaps most relevant for us in terms of understanding the range of validity of Theorem \ref{no atmen and t} is that amenable groupoids are always a-T-menable \cite[Lemme 3.5]{Tu:1999bq}.

We need one more technical condition for the proof.

\begin{definition}\label{gnc def}
A groupoid $G$ is \emph{large} if for any compact subset $K$ of $G$ the restriction of the range map $r|_{G\setminus K}:G\setminus K\to G^{(0)}$ is surjective.
\end{definition}

Note that a large groupoid is automatically non-compact; in general, one should think of largeness as a fairly mild generalisation of non-compactness.  For example, it is straightforward to see that a transformation groupoid $X\rtimes\Gamma$ (with $X$ compact) is large if and only if $\Gamma$ is not finite, if and only if $X\rtimes \Gamma$ is not compact.  We also have the following result: it implies in particular that largeness and non-compactness are equivalent for minimal groupoids.

\begin{lemma}\label{min large}
Let $G$ be a minimal, infinite groupoid.  Then $G$ is large.
\end{lemma}

\begin{proof}
Assume for contradiction that $G$ is minimal and infinite, but that there is a compact $K\subseteq G$ and $x\in G^{(0)}$ such that $x\not\in r(G\setminus K)$.  It follows that $G^x\subseteq K$.  As $K$ is compact, this forces $G^x$ to be finite, and thus the orbit of $x$ under $G$ must be finite.  This contradicts minimality unless $G^{(0)}$ equals the finite orbit of $x$, so in particular $G^{(0)}$ is finite and $G$ acts transitively on it.  However, as $G^x$ is finite and $G$ acts transitively on $G^{(0)}$, this forces each range fibre to be finite.  Hence $G$ is finite, which is the desired contradiction.
\end{proof}

Here is the main result of this section.  

\begin{theorem}\label{no atmen and t}
Let $G$ be an a-T-menable groupoid with property (T), and an invariant probability measure $\mu$ on $G^{(0)}$.  Then $G$ is not large.  
\end{theorem}

Note that this result together with Lemma \ref{min large} imply Theorem \ref{atmen t min}.  We discuss the failure of some stronger  statements in Example \ref{atmen t counterex} below.

In order to prove Theorem \ref{no atmen and t}, we need some basic facts about positive type functions (see Definition \ref{pt fun def}) on groupoids, and the associated GNS-type representations.  Let then $\mu$ be an invariant probability measure on $G^{(0)}$ and $\phi:G\to \C$ be a bounded, Borel, positive type function. We define an inner product on $C_c(G)$ by the formula
$$
\langle \xi,\eta\rangle_\phi:=\int_{G^{(0)}} \sum_{g,h\in G^x} \overline{\xi(g)}\eta(h)\phi(g^{-1}h)d\mu(x).
$$
The fact that $\phi$ is of positive type implies that the sum 
$$
\sum_{g,h\in G^x} \overline{\xi(g)}\xi(h)\phi(g^{-1}h)
$$
is non-negative for all $x\in G^{(0)}$, and thus that the form above is positive semidefinite.  Hence we may define a Hilbert space $H_\phi$ to be the separated completion of $C_c(G)$ for the above inner product.

The following lemma is presumably well-known (compare also Remark \ref{disint rem} below).  However, we could not find what we needed in the literature so give a proof for the reader's convenience.

\begin{lemma}\label{ccg act on hmu}
Let $\phi:G\to \C$ be a bounded, Borel, positive type function on a groupoid $G$, and let $\mu$ be an invariant probability measure on $G^{(0)}$.  Then with notation as above, the left convolution action of $C_c(G)$ induces a well-defined representation $\pi_\phi$ of $C_c(G)$ on $H_\phi$ by bounded operators.
\end{lemma}

\begin{proof}
Note that a general element $f\in C_c(G)$ is a finite sum of elements supported on bisections.  Hence to prove that the convolution action of a general $f\in C_c(G)$ on $H_\phi$ is well-defined and by bounded operators, it suffices to prove this for some $f\in C_c(G)$ supported on a single bisection.

Let $f,\xi\in C_c(G)$ with $f$ supported on a single bisection, and let us compute $\langle f\xi,f\xi\rangle_\phi$, where the product is convolution.  We have that 
$$
\langle f\xi,f\xi\rangle_\phi=\int_{G^{(0)}} \sum_{g,h,k,l\in G^x} \overline{f(k)\xi(k^{-1}g)}f(l)\xi(l^{-1}h)\phi(g^{-1}h)d\mu(x).
$$  
As $f$ is supported on a bisection, it can be non-zero on at most one point in $G^x$; hence we may replace the sums over $k$ and $l$ in the above by a single sum in $k$, getting 
$$
\langle f\xi,f\xi\rangle_\phi=\int_{G^{(0)}} \sum_{g,h,k\in G^x} |f(k)|^2\overline{\xi(k^{-1}g)}\xi(k^{-1}h)\phi(g^{-1}h)d\mu(x).
$$
Making the substitutions $m=k^{-1}g$ and $n=k^{-1}h$, we get
$$
\langle f\xi,f\xi\rangle_\phi=\int_{G^{(0)}} \sum_{k\in G^x}\sum_{m,n\in G^{s(k)}} |f(k)|^2\overline{\xi(m)}\xi(n)\phi(m^{-1}n)d\mu(x).
$$
Now the right hand side above is the integral of the function 
$$
G\to \C, \quad k\mapsto \sum_{m,n\in G^{s(k)}} |f(k)|^2\overline{\xi(m)}\xi(n)\phi(m^{-1}n)
$$
with respect to the measure $r^*\mu$.  Hence by invariance of $\mu$ it equals the same integral with respect to $s^*\mu$, i.e.\
\begin{align*}
\langle f\xi,f\xi\rangle_\phi & =\int_{G^{(0)}} \sum_{k\in G_x}\sum_{m,n\in G^{s(k)}} |f(k)|^2\overline{\xi(m)}\xi(n)\phi(m^{-1}n)d\mu(x) \\
&=\int_{G^{(0)}} \Big(\sum_{k\in G_x}|f(k)|^2\Big)\Big(\sum_{m,n\in G^{x}} \overline{\xi(m)}\xi(n)\phi(m^{-1}n)\Big)d\mu(x) 
\end{align*}
Note now that as $f$ is supported in a bisection, we have that $\sum_{k\in G_x}|f(k)|^2\leq \|f\|_\infty^2$.  Hence we now have that 
$$
\langle f\xi,f\xi\rangle_\phi \leq \|f\|_\infty^2\int_{G^{(0)}} \sum_{m,n\in G^{x}} \overline{\xi(m)}\xi(n)\phi(m^{-1}n)d\mu(x) =\|f\|^2_\infty\langle \xi,\xi\rangle_\phi.
$$
This proves both that the action of $f$ on $H_\phi$ is well-defined, and that it is by a bounded operator.

It remains to prove that $\pi_\phi$ is a $*$-representation.  Linearity is clear, and multiplicativity follows from associativity of multiplication on $C_c(G)$, so it remains to check that $\pi_\phi$ is $*$-preserving.  We compute that for $f,\xi,\eta\in C_c(G)$,
$$
\langle \xi,f\eta\rangle_\phi=\int_{G^{(0)}}\sum_{g,h,k\in G^x}\overline{\xi(g)}f(k)\eta(k^{-1}h)\phi(g^{-1}h)d\mu(x).
$$
Making the substitutions $l=k^{-1}h$ and $m=k^{-1}g$, this equals
$$
\int_{G^{(0)}}\sum_{k\in G^x}\sum_{l,m\in G^{s(k)}}\overline{\xi(km)}f(k)\eta(l)\phi(m^{-1}l)d\mu(x).
$$
Using invariance of $\mu$ again we have
\begin{align*}
\int_{G^{(0)}} & \sum_{k\in G_x}\sum_{l,m\in G^{s(k)}}\overline{\xi(km)}f(k)\eta(l)\phi(m^{-1}l)d\mu(x) \\ & =\int_{G^{(0)}}\sum_{k\in G_x}\sum_{l,m\in G^{x}}\overline{\xi(km)}f(k)\eta(l)\phi(m^{-1}l)d\mu(x).
\end{align*}
On the other hand, $f(k)=\overline{f^*(k^{-1})}$, so this becomes
$$
\int_{G^{(0)}}\sum_{l,m\in G^{x}}\sum_{k\in G_x}\overline{f^*(k^{-1})\xi(km)}\eta(l)\phi(m^{-1}l)d\mu(x).
$$
The sum $\sum_{k\in G_x}\overline{f^*(k^{-1})\xi(km)}$ is just the complex conjugate of the convolution product of $f^*$ and $\xi$ evaluated at $m$, however, so this equals 
$$
\int_{G^{(0)}}\sum_{l,m\in G^{x}}\overline{(f^*\xi)(m)}\eta(l)\phi(m^{-1}l)d\mu(x)=\langle f^*\xi,\eta\rangle_\phi
$$
and we are done.
\end{proof}

\begin{remark}\label{disint rem}
We could also have deduced the above lemma from general theory, at least in the case that $G$ is second countable.  Indeed, for each $x\in G^{(0)}$ we may define a positive definite sesquilinear form on $C_c(G^x)$ by the formula
$$
\langle \xi,\eta \rangle_{x}:=\sum_{g,h\in G^x} \overline{\xi(g)}\eta(h)\phi(g^{-1}h),
$$
and so a Hilbert space $H_x$.   We then equip the collection $H=\{H_x\}_{x\in G^{(0)}}$ with the fundamental space of sections given by the image of $C_c(G)$; if $G$ is second countable, this makes $H$ into a measurable field of Hilbert spaces in the sense of  \cite[Definition 1.3.12]{Renault:2009zr} (or equivalently, a Hilbert bundle in the sense of \cite[Definition 2.2]{Anantharaman-Delaroche:2005pb}, as already used in Definition \ref{rep} above).  

We then equip $H$ with a representation $\pi$ of $G$ in the sense of \cite[Definition 2.3.12]{Renault:2009zr}, or equivalently of Definition \ref{rep} above, by defining for each $g\in G$ 
$$
\pi_g:H_{s(g)}\to H_{r(g)},\quad (\pi_g\xi)(h):=\xi(g^{-1}h).
$$
The reader can then verify for themselves that the integrated form of this representation (see \cite[Section 2.3.3]{Renault:2009zr}) agrees with $(H_\phi,\pi_\phi)$ as defined above.  We instead went via Lemma \ref{ccg act on hmu} as this seemed a little more direct, and as it does not require any separability assumptions on $G$.
\end{remark}

Let us now go back to the assumptions of Theorem \ref{no atmen and t}.  As $G$ is a-T-menable with compact base space, there exists a continuous, proper negative type function $F:G\to [0,\infty)$ as in Definition \ref{neg type def}.  It follows from Schoenberg's theorem (see for example \cite[Theorem C.3.2]{Bekka:2000kx}) that for each $t>0$ the function 
$$
\phi_t:G\to \R,\quad g\mapsto e^{-tF(g)}
$$
is positive type.  Hence we may form the Hilbert spaces $H_t:=H_{\phi_t}$ and representations $\pi_t:=\pi_{\phi_t}$ of $C_c(G)$ as in the discussion above.  

\begin{lemma}\label{alm inv}
With notation as above, the representations $(H_t,\pi_t)$ have the following property.  Let $\xi$ denote the image in $H_t$ of the characteristic function of the base space.  For all $\epsilon>0$ and compact subsets $K$ of $G$ there exists $T>0$ such that for all $t\in (0,T]$ and all $f\in C_c(G)$ with $\|f\|_I\leq1$ we have that 
$$
\|f\xi-\Psi(f)\xi\|_{H_t}<\epsilon.
$$
\end{lemma}

\begin{proof}
We compute that 
\begin{align*}
\|f\xi-\Psi(f)\xi\|^2_{H_t} & =\langle f\xi-\Psi(f)\xi,f\xi-\Psi(f)\xi\rangle_{H_t} \\ & =\int_{G^{(0)}}\sum_{g,h\in G^x} \overline{(f\xi-\Psi(f)\xi)(g)}(f\xi-\Psi(f)\xi)(h)\phi_t(g^{-1}h)d\mu(x).
\end{align*}
Using that $\xi$ is the identity for convolution, this equals 
$$
\int_{G^{(0)}}\sum_{g,h\in G^x} \overline{f(g)}f(h)\big(\phi_t(g^{-1}h)-\phi_t(g^{-1})-\phi_t(h)+\phi_t(x)\big)d\mu(x).
$$
As $\mu$ is a probability measure, the absolute value of this is bounded above by 
\begin{align*}
\sup_{g,h\in K}|\phi_t(g^{-1}h) & -\phi_t(g^{-1})-\phi_t(h)+\phi_t(x)|\sum_{g,h\in G^x} |\overline{f(g)}f(h)| \\  & \leq \sup_{g,h\in K}|\phi_t(g^{-1}h)-\phi_t(g^{-1})-\phi_t(h)+\phi_t(x)|\|f\|_\infty \|f\|_I \\ & \leq \sup_{g,h\in K}|\phi_t(g^{-1}h)-\phi_t(g^{-1})-\phi_t(h)+\phi_t(x)|.
\end{align*}
As $K$ is compact and $\phi_t(g)=e^{-tF(g)}$, all four terms in the last expression can be made to be within $\epsilon/4$ of $1$ for $t$ suitably small (depending only on the fixed function $F$, and $K$ and $\epsilon$), so we have the result.
\end{proof}

We need one more ancillary lemma that will let us use largeness.

\begin{lemma}\label{large lem}
A groupoid $G$ is large if and only if for every compact subset $K$ of $G$ there exists $f\in C_c(G)$ with support in $G\setminus K$, values in $[0,1]$, and with $\Psi(f):G^{(0)}\to \C$ equal to the constant function with value one.
\end{lemma}

\begin{proof}
If $K\subseteq G$ is compact, and $f\in C_c(G)$ is a function as in the statement, then we have that 
$$
1=\Psi(f)(x)=\sum_{g\in G^x}f(g)=\sum_{g\in G^x\setminus K}f(g)
$$
for all $x\in G^{(0)}$.  Hence in particular $G^x\setminus K$ must be non-empty for each $x$, which is largeness.

Conversely, assume $G$ is large.  Then for each $x\in G^{(0)}$, we may choose an open bisection $B_x\subseteq G\setminus K$ such that $r(B_x)\owns x$.  As $G^{(0)}$ is compact, we may take a finite subcover $\{r(B_{x_1}),...,r(B_{x_n})\}$ of the cover $\{r(B_x)\mid x\in G^{(0)}\}$ of $G^{(0)}$.  Choose a partition of unity $\{\phi_i:G^{(0)}\to [0,1]\mid i\in \{1,...,n\}\}$ on $G^{(0)}$ such that $\phi_i^{(0)}$ has compact support contained in $r(B_{x_i})$.  Define moreover $f_i:G\to [0,1]$ by 
$$
f_i(g) = \left\{\begin{array}{ll} \phi_i(r(g)) & g\in B_{x_i} \\ 0 & \text{ otherwise}\end{array}\right.
$$
Then each $f_i$ is continuous and compactly supported.  Define finally $f:=\sum_{i=1}^n f_i$.  It is not too difficult to see that this $f$ has the properties required by the statement, so we are done.
\end{proof}

We are now ready to complete the proof of Theorem \ref{no atmen and t}.

\begin{proof}[Proof of Theorem \ref{no atmen and t}]
With notation as above, let us assume for contradiction that $G$ is a-T-menable, has property (T), that $G^{(0)}$ is equipped with an invariant probability measure $\mu$, and that $G$ is large.  To derive a contradiction, it will be sufficient to prove that no representation $(H_t,\pi_t)$ has a non-zero invariant vector.  Indeed, Lemma \ref{alm inv} then contradicts property (T). 

Let us then assume for contradiction that some $(H_t,\pi_t)$ does have an invariant unit vector, say $\xi$.  Choose $\eta\in C_c(G)$ so that $\|\xi-\eta\|_{H_t}<1/4$.  
Let 
\begin{equation}\label{m def}
m:=4\|\eta\|_I\|\eta\|_\infty,
\end{equation}

Let $N$ be the support of $\eta$ and choose a compact subset $K$ such that $|\phi_t(g)|<1/ m$ for all $g\in N^{-1}\cdot (G\setminus K)\cdot N$; this is possible by properness of $F$, by compactness of $N$, and by the fact that $\phi_t(g)=e^{-tF(g)}$ for all $g\in G$.  Let $f\in C_c(G)$ be as in the definition of largeness for this $K$.  

Now, on the one hand, using invariance of $\xi$ we get 
\begin{align}\label{lower bound}
|\langle \eta,f\eta\rangle|> |\langle \xi,f\xi\rangle|-2/4 =  |\langle \xi,\Psi(f)\xi\rangle| -1/2=\|\xi\|^2-1/2=1/2.
\end{align}
On the other hand, 
$$
\langle \eta,f\eta\rangle=\int_{G^{(0)}}\sum_{g,h,k\in G^x}\overline{\eta(g)}f(k)\eta(k^{-1}h)\phi_t(g^{-1}h)d\mu(x).
$$
For the expression $\overline{\eta(g)}f(k)\eta(k^{-1}h)\phi_t(g^{-1}h)$ to be non-zero, we must have that $k\in G\setminus K$, that $k^{-1}h$ is in $N$ and that $g^{-1}\in N^{-1}$, whence $h\in k\cdot N\subseteq (G\setminus K)\cdot N$, and so $g^{-1}h$ is in $N^{-1}\cdot (G\setminus K)\cdot N$; hence whenever this expression is non-zero, we have that $|\phi_t(g^{-1}h)|<1/m$.  It follows that 
$$
|\langle \eta,f\eta\rangle|\leq  \frac{1}{m}\int_{G^{(0)}}\sum_{g,h,k\in G^x}|\overline{\eta(g)}||f(k)\eta(k^{-1}h)|d\mu(x)
$$
The $\sum_{h,k\in G^x}|f(k)\eta(k^{-1}h)|$ is bounded above by $\|\eta\|_\infty\|f\|_I$, and the assumptions on $f$ imply that $\|f\|_I=1$.  Hence we get
$$
|\langle \eta,f\eta\rangle|\leq \frac{1}{m}\|\eta\|_\infty\int_{G^{(0)}}\sum_{g\in G^x}|\overline{\eta(g)}|d\mu(x).
$$
The expression $\int_{G^{(0)}}\sum_{g\in G^x}|\overline{\eta(g)}|d\mu(x)$ is bounded above by the $I$-norm of $\eta$, and thus by definition of $m$ (line \eqref{m def}) we get 
$$
|\langle \eta,f\eta\rangle|\leq 1/4.
$$
This contradicts line \eqref{lower bound}, however, completing the proof.
\end{proof}

We conclude this section with an example showing that the existence of an invariant probability measure is necessary on Theorem \ref{no atmen and t}, and that one cannot in general conclude that $G$ is compact under the same hypotheses (as opposed to the weaker conclusion that $G$ is not large).

\begin{example}\label{atmen t counterex}
Let $\Gamma$ be the free group on two generators acting on its Gromov, or ideal, boundary $X$; see for example \cite[Section 5.1]{Brown:2008qy} for a direct treatment of this.  Let $G=X\rtimes \Gamma$ be the associated transformation groupoid.  As $\Gamma$ is a-T-menable, it is not difficult to show that $G$ is a-T-menable.  Moreover, the action of $\Gamma$ on $X$ is amenable (see for example \cite[Section 5.1]{Brown:2008qy} again) whence $C^*_{\max}(G)=C^*_r(G)$ (see for example \cite[Corollary 5.6.17]{Brown:2008qy}).  It follows from this and the canonical identification $C^*_r(G)=C(X)\rtimes_r\Gamma$ that the natural inclusion $\C[\Gamma]\to C^*_{\max}(G)$ extends to an inclusion $C^*_r(\Gamma)\to C^*_{\max}(G)$.  This implies that the collection $\mathcal{U}_X$ of Definition \ref{ux def} consists of representations of $\Gamma$ that extend to $C^*_r(\Gamma)$.  As $\Gamma$ is not amenable, $\Gamma$ therefore has property (T) with respect to $\mathcal{U}_X$ in the sense of Definition \ref{iso def}.  Thanks to Proposition \ref{trans t}, we may conclude therefore that $X\rtimes \Gamma$ has property (T).

To summarise, if $G$ is the transformation groupoid associated to the action of the free group $F_2$ on its Gromov boundary, then $G$ is a-T-menable and has property (T); it is also large, as this is true for any transformation groupoid $X\rtimes \Gamma$ with $X$ compact and $\Gamma$ infinite.  Hence the existence of an invariant probability measure is needed in Theorem \ref{no atmen and t}.  Moreover, as is well-known (and not difficult to check directly from the description given in \cite[Section 5.1]{Brown:2008qy}), $G$ is a minimal groupoid, so the existence of an invariant probability measure is also necessary in Theorem \ref{atmen t min}.

We may also use this example to build a non-compact groupoid $G$ which is a-T-menable, has property (T), and for which there exists an invariant probability measure; thus we cannot get the stronger conclusion that $G$ is non-compact in Theorem \ref{no atmen and t}.  Indeed, let $\Gamma$ and $X$ be as before, let $\{pt\}$ be the trivial groupoid with base space a single point, and let $G=X\rtimes \Gamma \bigsqcup \{pt\}$ be the disjoint union with the natural groupoid operations.  Then using the discussion above it is not difficult to see that $G$ is a-T-menable and property (T).  It is not compact as $X\rtimes \Gamma$ is not compact, and it has an invariant probability measure given by the Dirac mass on the trivial point.  
\end{example}

\section{Kazhdan projections}\label{kaz sec}

Throughout this section, $G$ denotes a groupoid.  As usual, groupoids will be locally compact, Hausdorff, and \'{e}tale, and have compact base space: see Convention \ref{gpd conv}.

In this section we will use property (T) for a groupoid $G$ to construct so-called \emph{Kazhdan projections} in  $C^*_{\max}(G)$, and explore some connections to exactness properties of $C^*_r(G)$ and the Baum-Connes conjecture.  The analogous classical result in the group case is due to Akemann and Walter \cite{Akemann:1981dz}.  See also Valette's paper \cite{Valette:1984wy}: Theorem 3.2 from this paper is one motivation for our approach to constructing Kazhdan projections.  

Kazhdan projections are interesting partly as (other than in trivial cases) they give examples of projections in $C^*_{\max}(G)\setminus C_c(G)$; these projections are thus quite exotic in some sense, and exist for `analytic' as opposed to `algebraic' reasons.  

Another reason Kazhdan projections are interesting is due to their connections to the Baum-Connes conjecture and exactness.  This was exploited to great effect by Higson, Lafforgue, and Skandalis in their construction of counterexamples to the Baum-Connes conjecture \cite{Higson:2002la}; part of the motivation for what we do here is to try to better understand some of the ideas in their work.

In the group case, Kazhdan projections usually refer to projections living in the maximal groupoid $C^*$-algebra. Here, we also study the existence of Kazhdan type projections living in completions of $C_c(G)$ with respect to general families of representations. The extra generality causes no difficulties, and covers interesting examples.    

\begin{definition}\label{f c*alg}
Let $\mathcal{F}$ be a family of representations of $C_c(G)$.  The $C^*$-algebra $C^*_{\mathcal{F}}(G)$ is defined to be the separated completion of $C_c(G)$ for the (semi-)norm defined by
$$
\|f\|_{\mathcal{F}}:=\sup_{(H,\pi)\in \mathcal{F}}\|\pi(f)\|_{\mathcal{B}(H)}.
$$
\end{definition}

\begin{definition}\label{kaz proj def}
Let $G$ be an groupoid.  A projection $p\in C^*_{\mathcal{F}}(G)$ is a \emph{Kazhdan projection} if its image in any $*$-representation of $C^*_{\mathcal{F}}(G)$ is the orthogonal projection onto the constant vectors.
\end{definition}

Note that if it exists, a Kazhdan projection is uniquely determined by the defining condition, so we will just say `the' Kazhdan projection in future.  Note that the Kazhdan projection could exist and be zero: this happens if and only if $C^*_{\mathcal{F}}(G)$ does not have any $*$-representations with non-zero constant vectors.  For example, for $C^*_{\max}(G)$, Corollary \ref{inv meas} implies this happens if and only if $G^{(0)}$ does not admit an invariant probability measure. 

\begin{example}\label{kaz proj com}
Say $G$ is compact.  Then the function $p=\chi/(\Psi(\chi)\circ r)$ from the proof of Proposition \ref{com gpd t} is the Kazhdan projection in any $C^*_{\mathcal{F}}(G)$.
\end{example}

\subsection{Existence of Kazhdan projections}

Our first goal is to prove a general existence result for Kazhdan projections.  

To state it, recall that a groupoid $G$ is \emph{compactly generated} if there is a compact subset $K$ of $G$ such that any subgroupoid of $G$ containing $K$ must be all of $G$. 

\begin{theorem}\label{kaz proj}
Say $G$ is a compactly generated groupoid which has property (T) with respect to the family $\mathcal{F}$.  Then there exists a Kazhdan projection $p\in C^*_{\mathcal{F}}(G)$.
\end{theorem}

The proof will proceed via some lemmas.

\begin{lemma}\label{bis}
Say $G$ is a compactly generated groupoid that has property (T) with respect to a family of representations $\mathcal{F}$.   Then there exists a constant $c>0$ and a finite set $\phi_1,...,\phi_n$ of functions $G\to [0,1]$ supported on relatively compact open bisections such that the set 
\begin{equation}\label{gen set}
\bigcup_{i=1}^n \{g\in G\mid \phi_i(g)\geq 1/n\}
\end{equation}
generates $G$, and such that for any representation $(H,\pi)$ in $\mathcal{F}$ and any vector $\xi\in H_\pi$ we have that 
\begin{equation}\label{kaz inq}
\|(\phi_i-\Psi(\phi_i))\xi\|\geq c\|\xi\|
\end{equation}
for at least one $i$.
\end{lemma}

\begin{proof}
As $G$ is \'{e}tale and locally compact, it is covered by its open, relatively compact bisections.  Let $K\subseteq G$ be a compact set that is simultaneously a Kazhdan set for $\mathcal{F}$, and that generates $G$.   As $K$ is compact, it therefore admits a finite cover by relatively compact open bisections; let $\phi_1,...,\phi_n$ be a partition of unity subordinate to this open cover, so each $\phi_i$ takes values in $[0,1]$, is supported on some open relatively compact bisection, and for all $g\in K$, $\sum_{i=1}^n \phi_i(g)=1$.  We claim that $\phi_1,...,\phi_n$ have the required properties.  

Indeed, as $K$ generates $G$, the set in line \eqref{gen set} generates $G$ as it contains $K$.  To see the inequality in line \eqref{kaz inq}, note that as $K$ is a Kazhdan set there exists a constant $c_0>0$ such that for any representation $(H,\pi)$ in $\mathcal{F}$ and any vector $\xi\in H_\pi$ there exists $f\in C_c(G)$ supported in $K$ with $\|f\|_I\leq 1$ and such that 
\begin{equation}\label{base inq}
\|(f-\Psi(f))\xi\|\geq c_0\|\xi\|.
\end{equation}
Let $f_i:G^{(0)}\to \C$ be defined by 
$$
f_i(x):=\left\{\begin{array}{ll} f(g) & \text{ there is } g\in G^x\cap \text{supp}(\phi_i) \\ 0 & \text{otherwise}\end{array}\right.;
$$ 
as $G^x\cap \text{supp}(\phi_i)$ contains at most one point, this makes sense, and each $f_i$ is a bounded Borel function of compact support with $\|f_i\|_\infty\leq \|f\|_I\leq 1$.  Noting that the representation of $C(G^{(0)})\subseteq C_c(G)$ extends canonically to a representation of the $C^*$-algebra of bounded Borel functions on $G^{(0)}$, we may make sense of each $f_i$ as an operator on $H$, and we have the formula 
$$
f=\sum_{i=1}^n f_i\phi_i
$$
(where each product $f_i\phi_i$ means convolution of functions, or equivalently composition of operators) as operators on $H$.  Now, we have from line \eqref{base inq} that
\begin{align*}
c_0\|\xi\| & \leq \|(f-\Psi(f))\xi\|\leq \sum_{i=1}^n \|(f_i\phi_i-\Psi(f_i\phi_i))\xi\|= \sum_{i=1}^n \|f_i(\phi_i-\Psi(\phi_i))\xi\| \\ & \leq \sum_{i=1}^n \|f_i\|_\infty \|(\phi_i-\Psi(\phi_i))\xi\|\leq \sum_{i=1}^n \|(\phi_i-\Psi(\phi_i))\xi\|.
\end{align*}
The result follows with $c=c_0/n$.
\end{proof}

Now, with notation as in Lemma \ref{bis}, for each $i$, define 
$$
\Delta_i:=(\phi_i-\Psi(\phi_i))^*(\phi_i-\Psi(\phi_i)).
$$
Then clearly each $\Delta_i$ is an element of $C_c(G)$ whose image in any $*$-representation is a positive operator.  Define 
$$
\Delta:=\sum_{i=1}^n \Delta_i.
$$
One should think of $\Delta$ as a combinatorial Laplacian-type operator: indeed, it is an analogue of the well-studied group Laplacian for a discrete group with finite generating set $S$, defined by 
$$
\Delta_\Gamma:=\sum_{s\in S}2-s-s^*=\sum_{s\in S}(s-1)^*(s-1)\in \C[\Gamma].
$$

\begin{lemma}\label{gen}
With notation as above, for any representation $(H,\pi)$ of $C_c(G)$, the kernel of $\pi(\Delta)$ consists exactly of the constant vectors.
\end{lemma}

\begin{proof}
Fix a representation $(H,\pi)$ of $C_c(G)$; for simplicity we will omit $\pi$ from the notation.  Let $\xi$ be a constant vector in $H$.  Then $\phi_i\xi=\Psi(\phi_i)\xi$ for each $i$.  Hence $\xi$ is in the kernel of each $\Delta_i$, so in the kernel of $\Delta$.

Conversely, say $\xi$ is in the kernel of $\Delta$.  To show that $\xi$ is constant, it will suffice to show that for any $g\in G$, there is a non-negative function $f_g:G\to \C$ with compact support contained in a bisection, such that $f_g(g)=1$, and such that $f_g\xi=\Psi(f_g)\xi$.  Indeed, this suffices because any element of $f\in C_c(G)$ can then be written as a finite sum of products of the form $f=\sum_{i=1}^n \psi_if_{g_i}$ where $g_1,...,g_n\in G$, $f_{g_i}$ has the properties above, and $\psi_i$ is in $C(G^{(0)})$.   One then has that 
$$
\Psi(f)\xi=\sum_{i=1}^n \Psi(\psi_if_{g_i})\xi=\sum_{i=1}^n \psi_i\Psi(f_{g_i})\xi=\sum_{i=1}^n \psi_if_{g_i}\xi=f\xi,
$$ 
where the second equality uses that $\psi_i$ is supported on the base space, and that $f_{g_i}$ is supported on a bisection.  

Fix $g\in G$.  We aim to construct $f_g$ with the properties. above.  As the set in line \eqref{gen set} generates $G$, there exist $g_1,...,g_k$ in this set such that $g=g_k\cdots g_1$.  Say $n(m)$ is such that each $g_m$ is in $\{h\in G\mid \phi_{n(m)}(h)\geq 1/n\}$.  We claim that we can choose continuous functions $\psi_1,...,\psi_k: G^{(0)}\to [0,\infty)$ and open neighbourhoods $U_m$ and $V_m$ of $g_m$ with the following properties:
\begin{enumerate}[(i)]
\item for each $m\in \{1,...,k\}$, $\psi_m\phi_m$ is supported on $U_m$;
\item for each $m\in \{1,...,k\}$, $\psi_m\phi_m$ is constantly equal to one on $V_m$;
\item for each $m\in \{2,...,k\}$, $s(U_m)\subseteq r(V_{m-1})$.  
\end{enumerate} 

We will do this by an iterative construction starting with $m=1$.  Define $V_1:=\{h\in G\mid \phi_{n(1)}(h)> 1/(2n)\}$, and $U_1:=\{h\in G\mid \phi_{n(1)}(h)>0\}$, so $g_1\in U_1\subseteq \overline{U_1}\subseteq V_1$ and both $U_1$ and $V_1$ are open bisections.  Let $\psi_1\in C(G^{(0)})$ be any non-negative function that is supported in $r(U_1)$, and is such that $\psi_1(r(h))=1/\phi_{n(1)}(h)$ for all $h\in V_1$ (this formula makes sense as $\phi_{n(1)}$ is supported on a bisection).  Assume now that we have chosen $\psi_1,...,\psi_{m-1}$,  $U_1,..,U_{m-1}$, and $V_1,...,V_{m-1}$ with the right properties.  Let $U_{m}$ be any open bisection containing $g_{m}$ and such that $s(U_m)\subseteq r(V_{m-1})$.  Choose an open set $V_m$ such that $V_m\owns g_m$ and such that
$$
\overline{V_m}\subseteq U_m\cap \{h\in G\mid \phi_{n(m)}(h)>1/(2n)\}
$$
Choose $\psi_{m}\in C(G^{(0)})$ to be any function that is supported in $r(U_m)$, and such that $\psi_{m}$ satisfies $\psi_{m}(r(h))=1/\phi_{n(m)}(h)$ on $r(V_{m})$.  This completes the iterative construction.

We claim that the function 
\begin{equation}\label{fg form}
f_g:=\psi_k\phi_{n(k)}\cdots \psi_1\phi_{n(1)}
\end{equation}
has the required properties.  Indeed, note that $f_g$ is non-negative and has compact support contained in a bisection as each $\phi_i$ has the same properties.  Moreover, $f_g(g)=1$ as $(\psi_m\phi_{n(m)})(g_m)=1$ for all $m$.  It remains to show that $f_g\xi=\Psi(f_g)\xi$.  We prove this by induction on the number of elements $k$ appearing in an expression as in line \eqref{fg form} (and with the properties from our iterative construction above).  

Note first that  
$$
0=\langle \xi,\Delta\xi\rangle=\sum_{i=1}^n\|(\phi_i-\Psi(\phi_i))\xi\|^2
$$
whence $\phi_i\xi=\Psi(\phi_i)\xi$ for each $i$.   It follows that for any $m$
\begin{equation}\label{base case}
\psi_m\phi_{n(m)}\xi=\psi_m\Psi(\phi_{n(m)})\xi=\Psi(\psi_m\phi_{n(m)})\xi,
\end{equation}
where the last equality uses that $\psi_m$ is supported on $G^{(0)}$ and that $\phi_{n(m)}$ is supported on a bisection.  In particular, this completes the base case.  For the inductive step, say we know that 
$$
\psi_m\phi_{n(m)}\cdots \psi_1\phi_{n(1)}\xi=\Psi(\psi_m\phi_{n(m)}\cdots \psi_1\phi_{n(1)})\xi
$$
for some $m$.  Then using line \eqref{base case}
$$
\Psi(\psi_{m+1}\phi_{n(m+1)})\xi =\psi_{m+1}\phi_{n(m+1)}\xi.
$$
On the other hand, using that $\Psi(\psi_m\phi_{n(m)}\cdots \psi_1\phi_{n(1)})$ is constantly equal to one on $r(V_m)$ (this follows from the properties of the iterative construction), that $\psi_{m+1}\phi_{n(m+1)}$ is supported in $U_{m+1}$ and that $s(U_{m+1})\subseteq  r(V_m)$, we have
$$
\psi_{m+1}\phi_{n(m+1)}=\psi_{m+1}\phi_{n(m+1)}\Psi(\psi_m\phi_{n(m)}\cdots \psi_1\phi_{n(1)})
$$
and also that 
$$
\Psi(\psi_{m+1}\phi_{n(m+1)})=\Psi(\psi_{m+1}\phi_{n(m+1)}\psi_m\phi_{n(m)}\cdots \psi_1\phi_{n(1)}).
$$
Combining the last four displayed lines gives that 
\begin{align*}
\Psi(\psi_{m+1}\phi_{n(m+1)}\psi_m\phi_{n(m)}\cdots \psi_1\phi_{n(1)})\xi & =\Psi(\psi_{m+1}\phi_{n(m+1)})\xi \\ 
& =\psi_{m+1}\phi_{n(m+1)}\xi \\
& =\psi_{m+1}\phi_{n(m+1)}\Psi(\psi_m\phi_{n(m)}\cdots \psi_1\phi_{n(1)})\xi \\
& =\psi_{m+1}\phi_{n(m+1)}\psi_m\phi_{n(m)}\cdots \psi_1\phi_{n(1)}\xi,
\end{align*}
which completes the proof.
\end{proof}

\begin{lemma}\label{gap}
With $c>0$ as in Lemma \ref{bis}, we have that for any representation of $(H,\pi)$ in $\mathcal{F}$, the spectrum of $\pi(\Delta)$ is contained in $\{0\}\cup[c^2,\infty)$.
\end{lemma}

\begin{proof}
We have already seen that the kernel of $\pi(\Delta)$ consists precisely of $H^\pi$, so it suffices to show that $\langle \Delta\xi,\xi\rangle\geq c^2\|\xi\|^2$ for all $\xi\in H_\pi$.  Indeed, this follows directly from Lemma \ref{bis} as we have
\begin{align*}
\langle \Delta\xi,\xi\rangle =\sum_{i=1}^n \|(\phi_i-\Psi(\phi_i))\xi\|^2\geq c^2\|\xi\|^2
\end{align*}
as required.
\end{proof}

Putting the above together, we may now complete the proof of Theorem \ref{kaz proj}.

\begin{proof}[Proof of Theorem \ref{kaz proj}]
With $\Delta$ as above, Lemma \ref{gap} implies that the spectrum of $\Delta$ as an element of $C^*_{\mathcal{F}}(G)$ is contained in $\{0\}\cup[c^2,\infty)$.  Hence the characteristic function of zero $\chi_{\{0\}}$ is continuous on the spectrum of $\Delta$, and so we may set $p:=\chi_{\{0\}}(\Delta)\in C^*_{\mathcal{F}}(G)$.  This has the right property by Lemma \ref{gen}.
\end{proof}

\subsection{Kazhdan projections in $C^*_r(G)$ and exactness}

In this subsection, we want to study the Kazhdan projection in $C^*_r(G)$ when it exists.  In particular we aim to characterise when it is non-zero.  For this, we need to know when $C^*_r(G)$ has representations with non-zero constant vectors. 

The next lemma is the key technical ingredient.  To state it, recall from Example \ref{reg reps} that if $G$ is a groupoid and $x\in G^{(0)}$, then the \emph{regular representation} of $C_c(G)$ is defined to be the pair $(\ell^2(G_x),\pi_x)$, where for $f\in C_c(G)$, $\pi_x(f)$ acts via the usual convolution formula
$$
(\pi_x(f)\xi)(g):=\sum_{h\in G^{r(g)}}f(h)\xi(h^{-1}g).
$$ 

\begin{lemma}\label{pix con}
Let $G$ be a groupoid and let $(\ell^2(G_x),\pi_x)$ be the regular representation associated to some $x\in G^{(0)}$.  Then the invariant vectors in $\ell^2(G_x)$ in the sense of Definition \ref{con vec def} are exactly the functions $\xi:G_x\to \C$ that are constant in the usual sense.
\end{lemma}

\begin{proof}
It is straightforward to check that a constant function in $\ell^2(G_x)$ is invariant; we leave this to the reader. 

Conversely, let $\xi\in \ell^2(G_x)$ be a norm one invariant vector.  Then the associated probability measure $\mu$ on $G^{(0)}$ defined on $f\in C_c(G^{(0)})$ by 
$$
\mu(f):=\langle \xi,f\xi\rangle=\sum_{g\in G_x}|\xi(g)|^2f(r(g))
$$ 
is invariant by Proposition \ref{inv vec char}.  The measure $\mu$ equals the weighted sum $\sum_{g\in G_x}|\xi(g)|^2\delta_{r(g)}$ of Dirac masses.  Consider the orbit $Gx:=r(G^x)$, and define 
$$
w:Gx\to [0,1],\quad y\mapsto \sum_{g\in G_x\cap G^y} |\xi(g)|^2,
$$
so we have $\mu=\sum_{y\in Gx} w(y)\delta_y$.  As $\mu$ is invariant, we have that 
$$
\int_{G}fdr^*\mu=\int_{G}fds^*\mu,
$$
or in other words that 
\begin{equation}\label{inv meas 2}
\sum_{y\in Gx} w(y) \sum_{h\in G_{y}}f(h) = \sum_{y\in Gx} w(y)\sum_{h\in G^{y}}f(h) 
\end{equation}
for all $f\in C_c(G)$.  Now, for $y,z\in Gx$, fix $h\in G$ with $s(h)=y$ and $r(h)=z$. Let $(U_i)_{i\in I}$ be a basis of open neighbourhoods of $h$ with compact closure, and $(f_i)$ be a net of functions in $C_c(G)$ such that the support of $f_i$ is contained in $U_i$, $0\leq f_i \leq 1$, and $f_i(h)=1$ for all $i$.  
Then substituting $f_i$ into line \eqref{inv meas 2} above and taking the limit over $i$ forces $w(y)=w(z)$, or in other words that $w$ is constant on the orbit $Gx$.  As $\mu$ is a probability measure, this is impossible unless $Gx$ is finite.

Summarising, then: at this point, we have that the cardinality $n$ of $Gx$ is finite, and for each $y\in Gx$, 
$$
\mu(\{y\})=1/n.
$$
As the set $Gx$ is finite (and as $G$ is \'{e}tale), for each $g\in G_x$ there exists a continuous function $f:G\to [0,1]$ supported on a relatively compact open bisection such that  $f(g)=1$ and $f(h)=0$ for all $h\in G_x\setminus \{g\}$.  As $\xi$ is constant, we have that 
$$
\pi_x(f)\xi=\pi_x(\Psi(f))\xi
$$
as functions on $G_x$, and evaluating both sides of the line above at $g$ gives $\xi(x)=\xi(g)$.  This just says that $\xi$ is constant (in the naive sense of `taking the same value at each point of $G_x$'), so we are done.
\end{proof}

The following result gives a fairly precise characterisation of what the Kazhdan projection `looks like' in $C^*_r(G)$.  To state it, let $E:C^*_r(G)\to C(G^{(0)})$ be the canonical conditional expectation of \cite[Proposition 2.3.22]{Renault:2009zr}. 

\begin{proposition}\label{kaz im}
Let $G$ be a groupoid, and assume that there exists a Kazhdan projection $p\in C^*_{r}(G)$.  Then 
$$
\{x\in G^{(0)}\mid E(p)(x)>0\}=\{x\in X\mid G_x \text{ is finite } \}.
$$
In particular, if a Kazhdan projection $p$ exists in $C^*_{r}(G)$, then it is non-zero if and only if the source fibre $G_x$ is finite for some $x\in G^{(0)}$.  
\end{proposition}

\begin{proof}
Let $x\in G^{(0)}$, and let $(\ell^2(G_x),\pi_x)$ be the associated regular representation.  As in the proof of \cite[Proposition 2.3.20]{Renault:2009zr}, we have that 
$$
E(a)(x)=\langle \delta_x,\pi_x(a)\delta_x\rangle_{\ell^2(G_x)}
$$
for any $a\in C^*_r(G)$.  Using Lemma \ref{pix con}, the Kazhdan projection $\pi_x(p)$ is non-zero if and only if $G_x$ is finite, in which case its image consists of all constant vectors.  Note moreover that if $G_x$ is finite, then this description gives that 
$$
\langle \delta_x,\pi_x(p)\delta_x\rangle=\frac{1}{|G_x|}.
$$   
Hence we have that 
$$
E(p)(x)=\left\{\begin{array}{ll} 1/|G_x| & G_x \text{ finite} \\ 0 & G_x \text{ infinite} \end{array}\right.
$$
The given equality of sets follows.

The remaining statement follows as the canonical conditional expectation $E:C^*_r(G)\to C(G^{(0)})$ is faithful (see \cite[Proposition 2.3.22]{Renault:2009zr}).
\end{proof}

We now turn to an application to (inner) exactness.   Recall first that if $G$ is a groupoid, a subset $E$ of $G^{(0)}$ is \emph{invariant} if whenever $g\in G$ is such that $s(g)$ is in $E$, we also have that $r(g)$ is in $E$.  If $E$ is an open or closed invariant subset of $G^{(0)}$, then the \emph{restriction} $G|_E:=G_E^E$ it itself a (locally compact, Hausdorff, \'{e}tale) groupoid.  It follows directly from the definition of the reduced groupoid $C^*$-algebra (see for example \cite[Section 2.3.4]{Renault:2009zr}) that if $F$ is a closed invariant subset of $G^{(0)}$, then the natural restriction map $C_c(G)\to C_c(G|_F)$ extends to quotient $*$-homomorphism $C^*_r(G)\to C^*_r(G|_F)$.  Moreover, if $U$ is an open invariant subset of $G$, then $C_c(G|_U)$ is an ideal in $C_c(G)$, and the inclusion $C_c(G|_U)\to C_c(G)$ extends to an inclusion of a $C^*$-ideal $C^*_r(G|_U)\to C^*_r(G)$.  

If now $F$ is a closed invariant subset of $G^{(0)}$ and $U$ its (necessarily open and invariant complement), then in the diagram below
$$
\xymatrix{ 0 \ar[r] & C^*_r(G|_U) \ar[r]^-\iota &  C^*_r(G) \ar[r]^-\pi &  C^*_r(G|_F) \ar[r] & 0 }
$$
all the conditions needed to be a short exact sequence are always satisfied, except one may have that the kernel of $\pi$ strictly contains the image of $\iota$.  While it is often true that this sequence will be exact, this need not always be the case, leading to the next definition.

\begin{definition}\label{inner exact}
A groupoid $G$ is \emph{inner exact} if for any open invariant subset $U$ of $G^{(0)}$ with closed complement $F$, the canonical sequence
$$
\xymatrix{ 0 \ar[r] & C^*_r(G|_U) \ar[r] &  C^*_r(G) \ar[r] &  C^*_r(G|_F) \ar[r] & 0 }
$$
discussed above is exact in the middle.  
\end{definition}

Although it looks a little technical at first, the proposition below (combined with Theorem \ref{kaz proj}) gives many examples of non-inner exact groupoids coming from property (T).  It, or variations of it, underlies many of the counterexamples to the Baum-Connes conjecture considered in \cite{Higson:2002la}.  We do not claim, however, that the result is optimal in any sense.

\begin{proposition}\label{t ex}
Say $G$ is a groupoid such that the Kazhdan projection exists in $C^*_r(G)$.  Assume moreover that there is a closed invariant subset $F$ of $G^{(0)}$ with complement $U=G^{(0)}\setminus F$ and a net $(x_i)$ in $U$ with the following properties:
\begin{enumerate}[(i)]
\item for every $x\in F$, $G_x$ is infinite;
\item for every $i$, $G_{x_i}$ is finite;
\item for any compact subset $K$ of $U$, the orbit $Gx_i:=\{r(g)\mid g\in G_{x_i}\}$ does not intersect $K$ for all suitably large $i$.
\end{enumerate}
Then the sequence
$$
\xymatrix{ 0 \ar[r] & C^*_r(G|_U) \ar[r] &  C^*_r(G) \ar[r] &  C^*_r(G|_F) \ar[r] & 0 }
$$
is not exact, and in particular $G$ is not inner exact.
\end{proposition} 

\begin{proof}
With assumptions as in the proposition, note that the Kazhdan projection in $C^*_r(G)$ has to map to the Kazhdan projection in $C^*_r(G|_F)$, which is zero by the assumption that $F^{(0)}$ contains no points with finite source fibre, and Proposition \ref{kaz im}.  Thus we must show that $p$ is not in $C^*_r(G|_U)$; assume for contradiction that this is the case, so in particular there exists $a\in C_c(G|_U)$ such that $\|p-a\|_{C^*_r(G)}<1/2$.  

Let $(x_i)$ be the net in the assumptions.  Then as each $G_{x_i}$ is finite, Proposition \ref{kaz im} implies that the image $\pi_{x_i}(p)$ under the regular representation $\pi_{x_i}$ is a non-trivial projection, so norm one.  On the other hand, the assumption that the orbits $Gx_i$ are eventually disjoint from any compact subset of $U$ implies that $\pi_{x_i}(a)=0$ for all suitably large $i$.  Thus we have 
$$
1/2>\|p-a\|_{C^*_r(G)}\geq \limsup_i \|\pi_{x_i}(p)-\pi_{x_i}(a)\|=1,
$$
which is the desired contradiction.  
\end{proof}

\begin{examples}\label{exact ex}
There are two interesting examples where Proposition \ref{t ex} applies that we have discussed already in this paper; no doubt other examples are possible, but we will content ourselves with these here.

The first occurs for HLS groupoids (Definition \ref{hls gpd}), associated to a group and approximating sequence with property $(\tau)$ as in Proposition \ref{hls tau}.  In this case one can take $U$ to be the subset $\N$ of the unit space $\N\cup\{\infty\}$, and $F$ to be the singleton $\{\infty\}$.  

A second interesting example occurs when $X$ is an expander as in Definition \ref{exp def}.  Then $G(X)$ has property (T) with respect to the singleton family $\mathcal{F}_{\ell^2(X)}$ consisting of the natural representation on $\ell^2(X)$ by Proposition \ref{exp t}.  We have that $C^*_{\mathcal{F}}(G(X))$ equals $C^*_r(G(X))$ in this case (see for example \cite[Proposition 10.29]{Roe:2003rw}), so the Kazhdan projection exists in $C^*_r(G(X))$ by Theorem \ref{kaz proj}.  In this case, recall that $G(X)^{(0)}$ is the Stone-\v{C}ech compactification of $X$.  One can take $U$ to be $X$, and $F$ to be the Stone-\v{C}ech remainder $\beta X\setminus X\subseteq G(X)^{(0)}$.  
\end{examples}

\subsection{Kazhdan projections as $K$-theory classes}

In this subsection, we say a little about the class of the Kazhdan projection in $K$-theory.  We start with a discussion of failures of inner $K$-exactness.  

\begin{definition}\label{in k ex}
A groupoid $G$ is \emph{inner $K$-exact} if for every open invariant subset $U\subseteq G^{(0)}$ with closed complement $F$, the corresponding sequence 
$$
K_*(C^*_r(G|_U))\to K_*(C^*_r(G))\to K_*(C^*_r(G|_F))
$$
of $K$-theory groups is exact in the middle.  
\end{definition}

\begin{proposition}\label{in k ex res}
Under the assumptions of Proposition \ref{t ex}, the class $[p]\in K_0(C^*_r(G))$ of the Kazhdan projection goes to zero in $K_0(C^*_r(G|_F))$, but is not in the image of the map $K_0(C^*_r(G|_U))\to K_0(C^*_r(G))$.  In particular, $G$ fails to be inner $K$-exact.
\end{proposition}

\begin{proof}
We have seen that $p$ itself goes to zero in $C^*(G|_F)$, so it suffices to show that $[p]$ is not in the image of the map $K_*(C^*_r(G|_U))\to K_*(C^*_r(G))$.  Assume for contradiction that it is, so there exists some projection $q\in M_n(\widetilde{C^*_r(G|_U)})$ and $k\leq n$ such that $[p]=[1_k]-[q]$ in $K_0(C^*_r(G))$ (here $\widetilde{\cdot}$ denotes unitisation, and $1_k$ denotes the idempotent in $M_n(\C)$ with $k$ ones down the main diagonal, followed by $n-k$ zeros), and such that $q=1_k+a$ for some self-adjoint $a\in M_n(C^*_r(G|_U))$.  Let $b\in M_n(C_c(G_U))$ be self-adjoint and such that $\|a-b\|<1/100$.  

For each $i$, let $\pi_{x_i}:C^*_r(G)\to \mathcal{B}(\ell^2(G_x))$ denote the regular representation, where $(x_i)$ is the net in the assumptions.  Then for each $i$, the class $[\pi_{x_i}(p)]\in K_0(\mathcal{B}(\ell^2(G_x)))\cong \Z$ corresponds to the generator $1$, as $\pi_{x_i}(p)$ is a rank one projection by Proposition \ref{kaz im}.  On the other hand, if $\chi_{(1/2,\infty)}$ denotes the characteristic function of this interval, then the fact that $\|q-(1_k+b)\|<1/100$ implies that $\chi_{(1/2,\infty)}$ is continuous on the spectrum of $\pi_{x_i}(1_k+b)$ and moreover that 
\begin{align*}
1 & =[\pi_{x_i}(p)]=[\pi_{x_i}(1_k)]-[\chi_{(1/2,\infty)}(\pi_{x_i}(1_k+b)] \\ & =[1_k]-[\chi_{(1/2,\infty)}(\pi_{x_i}(1_k)+\pi_{x_i}(b))].
\end{align*}
As $b$ is compactly supported, the assumption that the orbits $Gx_i$ eventually do not intersect any compact subset of $U$ implies that $\pi_{x_i}(b)$ is zero for all suitably large $i$.  Thus the above displayed line implies that for all suitably large $i$, $1=0$ in $\Z$, giving the desired contradiction. 
\end{proof}

Combining this with the observation of Higson, Lafforgue, and Skandalis that a groupoid satisfying the Baum-Connes conjecture must be inner $K$-exact (see \cite[Section 1]{Higson:2002la}), we get the following corollary.

\begin{corollary}\label{bc cor}
Under the assumptions of Proposition \ref{t ex}, the Baum-Connes conjecture (with trivial coefficients) must fail for at least one of the groupoids $G|_U$, $G$, or $G|_F$. \qed
\end{corollary}

Another interesting connection to the Baum-Connes conjecture is given by the following result, saying that the class of the Kazhdan  projection cannot be in the image of the maximal Baum-Connes assembly map in some cases.  Variants of this are well-known for groups: see for example \cite[Section 5]{Higson:1998qd}.

\begin{lemma}\label{kaz bc}
Let $G$ be a groupoid such that the Kazhdan projection $p$ exists in $C^*_{\max}(G)$, and such that no source fibre is finite.  Assume moreover that the class $[p]\in K_0(C^*_{\max}(G))$ is non-zero, and that $G$ satisfies the Baum-Connes conjecture.  Then $[p]$ is not in the image of the maximal Baum-Connes assembly map. 
\end{lemma}

\begin{proof}
We have a commutative diagram 
$$
\xymatrix{ K_*^{top}(G) \ar[r]^-{\mu_m} \ar[dr]_-{\mu_r} & K_*(C^*_{\max}(G)) \ar[d]^-{\lambda_*} \\ & K_*(C^*_r(G)) }
$$
where the maps labeled $\mu_m$ and $\mu_r$ are respectively the maximal and reduced Baum-Connes assembly maps, and the map labeled $\lambda_*$ is the map on $K$-theory induced by the canonical quotient $\lambda:C^*_{\max}(G)\to C^*_r(G)$.  We are assuming that $\mu_r$ is an isomorphism, and Proposition \ref{kaz im} plus the assumption that no source fibre in $G$ is finite implies that the image of $[p]\in K_0(C^*_{\max}(G))$ under $\lambda_*$ is zero.  The result follows as we are assuming that $[p]\in K_0(C^*_{\max}(G))$ is non-zero.
\end{proof}

It would be interesting if one could show that $[p]\in K_0(C^*_{\max}(G))$ cannot be in the image of the maximal assembly map, \emph{without} assuming that $G$ satisfies the Baum-Connes conjecture; this is known for discrete groups \cite[Section 5]{Higson:1998qd}.

It would be also be interesting to have a good characterisation of when $[p]\neq 0$ in $K_*(C^*_{\max}(G))$; this is automatic in the group case, but we do not have a good general condition.  We do at least have the following observation.  It is implicit in the proof of Proposition \ref{in k ex res}, but it seemed potentially useful to make it explicit.

\begin{lemma}\label{kaz ne0}
Say $G$ is a groupoid, and assume the Kazhdan projection is not zero in $C^*_{\red}(G)$.  Then $[p]\neq 0$ in $K_0(C^*_{\red}(G))$.
\end{lemma}

\begin{proof}
Proposition \ref{kaz im} implies that there is some $x\in G^{(0)}$ with $G_x$ finite, and Lemma \ref{pix con} implies that $\pi_x(p)\neq 0$.  As $\mathcal{B}(\ell^2(G_x))$ is finite dimensional, all non-trivial projections in this algebra have non-zero $K_0$ class.  Hence the map $(\pi_x)_*:K_0(C^*_{\red}(G))\to K_0(\mathcal{B}(\ell^2(G_x))$ sends $[p]\in K_0(C^*_{\red}(G))$ to something non-zero, and so $[p]$ itself is non-zero.
\end{proof}

\section{Questions}\label{q sec}

We conclude the paper by summarising some open problems that we think are interesting.  Some of these we thought about and could not make progress with; others we did not attempt to address here mainly to keep the paper to a reasonable length (and would be more than happy for someone else to take up).

\begin{enumerate}[(i)]
\item Does property (T) for a groupoid $G$ imply some sort of fixed point property for affine actions on bundles of Hilbert spaces over $G^{(0)}$, analogous to the classical Delorme-Guichardet theorem for groups (see \cite[Chapter 2]{Bekka:2000kx})?
\item What (if any) is the precise relationship between our property (T), and the Dong-Ruan property (T) from Definition \ref{t2 def} above?
\item (Suggested by Jesse Peterson) Is there any connection between our property (T) and Bekka's definition \cite{Bekka:2006iz} of property (T) for (pairs of) $C^*$-algebras?
\item Is property (T) Morita invariant?
\end{enumerate}

\bibliography{Generalbib}

\end{document}